\documentclass[11pt]{amsart}

\usepackage{amsfonts}
\usepackage{amsmath}
\usepackage{amsthm}
\usepackage{graphicx}
\usepackage{caption}
\usepackage{natbib}
\usepackage{color}
\usepackage{hyperref}
\usepackage{setspace}
\usepackage{natbib}
\usepackage[T1]{fontenc}
\usepackage{times}
\usepackage[left=2cm,right=2cm,top=2.5cm,bottom=2.5cm]{geometry}

\usepackage{bigints}

\numberwithin{equation}{section}

\allowdisplaybreaks

\theoremstyle{plain}
\newtheorem{theorem}{Theorem}
\newtheorem{lemma}[theorem]{Lemma}

\newtheorem{corollary}[theorem]{Corollary}
\theoremstyle{definition}
\newtheorem{definition}[theorem]{Definition}
\newtheorem{remark}[theorem]{Remark}

\def\beqn{\begin{equation}}
\def\eeqn{\end{equation}}

\newcommand{\remove}[1]{}

\newcommand{\BX}{{\bf X}}

\newcommand{\bt}{{\bf t}}

\def\P{\mathbb{P}}
\def\E{\mathbb{E}}
\def\V{\mathbb{VAR}}
\def\Pn{\mathcal{P}_n}
\def\cP{\mathcal{P}}

\newcommand{\reals}{{\mathbb R}}
\newcommand{\bbr}{\reals}
\newcommand{\bbn}{{\mathbb N}}

\newcommand{\X}{{\mathcal{X}}}

\newcommand{\md}{\mathrm{d}}

\newcommand{\one}{{\bf 1}}

\newcommand{\ta}{\theta}
\newcommand{\Tngamma}{S_n^{(\gamma)}}
\newcommand{\Ungamma}{U_n^{(\gamma)}}
\newcommand{\sumXkPn}{\sum_{(X_1,\dots,X_k) \in \mathcal P_{n,\neq}^k}}

\newcommand{\barrhona}{\bar \rho_{n,\alpha}}
\newcommand{\angamma}{a_n^{(\gamma)}}
\newcommand{\Bdalpha}{B_d^{(\alpha)}}
\newcommand{\Bdzeta}{B_d^{(\zeta)}}

\newcommand{\ingam}{\int_0^{\gamma R_n}}
\newcommand{\Dk}{d_{(k)}}
\newcommand{\Di}{d_{(i)}}
\newcommand{\Done}{d_{(1)}}

\newcommand{\gngamma}{g_{n,\gamma}}
\newcommand{\hngamma}{h_{n,\gamma}}
\newcommand{\rhona}{\rho_{n,\alpha}}
\newcommand{\C}{\mathcal C}
\newcommand{\BT}{{\bf T}}
\newcommand{\Ta}{\Theta}

\newcommand{\red}[1]{\textcolor{red}{\textnormal{#1}}}

\numberwithin{theorem}{section}

\begin{document}




\title[Hyperbolic Random Geometric Graph]{Sub-tree counts on hyperbolic random geometric graphs}
\author{Takashi Owada}
\address{(TO) Department of Statistics\\
Purdue University \\
West Lafayette, 47907, USA}
\email{owada@purdue.edu}

\author{D. Yogeshwaran}
\address{(DY) Theoretical Statistics and Mathematics unit\\
Indian Statistical Institute \\
Bangalore, India}
\thanks{DY's research was supported in part by DST-INSPIRE Faculty fellowship and CPDA from the Indian Statistical Institute.}
\email{d.yogesh@isibang.ac.in}

\subjclass[2010]{Primary : 60F05, 60D05 ; Secondary : 05C80, 51M10}
\keywords{Hyperbolic spaces, Random geometric graphs, Poisson point process, sub-tree counts, central limit theorem, Malliavin-Stein method.}

\normalfont

\begin{abstract}
In this article, we study the hyperbolic random geometric graph introduced recently in \citep{krioukov:papadopoulos:kitsak:vahdat:boguna:2010}. For a sequence $R_n \to \infty$, we define these graphs to have the vertex set as Poisson points distributed uniformly in balls $B(0,R_n) \subset \Bdalpha$, the $d$-dimensional Poincar\'e ball (i.e., the unit ball on $\bbr^d$ with the Poincar\'e metric $d_{\alpha}$ corresponding to negative curvature $-\alpha^2, \alpha > 0$) by connecting any two points within a distance $R_n$ according to the metric $d_{\zeta}, \zeta > 0$. Denoting these graphs by $HG_n(R_n ; \alpha, \zeta)$, we study asymptotic counts of copies of a fixed tree $\Gamma_k$ (with the ordered degree sequence $d_{(1)} \leq \ldots \leq d_{(k)}$) in $HG_n(R_n ; \alpha, \zeta)$. Unlike earlier works, we count more involved structures, allowing for $d > 2$, and in many places, more general choices of $R_n$ rather than $R_n = 2[\zeta (d-1)]^{-1}\log (n/ \nu), \nu \in (0,\infty)$. The latter choice of $R_n$ for $\alpha / \zeta > 1/2$ corresponds to the thermodynamic regime in which the expected average degree is asymptotically constant. We show multiple phase transitions in $HG_n(R_n ; \alpha, \zeta)$ as $\alpha / \zeta$ increases, i.e., the space $\Bdalpha$ becomes more hyperbolic. In particular, our analyses reveal that the sub-tree counts exhibit an intricate dependence on the degree sequence $d_{(1)},\ldots,d_{(k)}$ of  $\Gamma_k$ as well as the ratio $\alpha/\zeta$. 
Under a more general radius regime $R_n$ than that described above, we investigate the asymptotics of the expectation and variance of sub-tree counts. Moreover, we prove the corresponding central limit theorem as well. Our proofs rely crucially on a careful analysis of the sub-tree counts near the boundary using Palm calculus for Poisson point processes along with estimates for the hyperbolic metric and measure. For the central limit theorem, we use the abstract normal approximation result from \citep{last:peccati:schulte:2016} derived using the Malliavin-Stein method. 

\end{abstract}

\date{\today}
\maketitle
%
%
\section{{\bf Introduction}}
\label{sec:intro}

In this article, we shall continue the study of random geometric graphs on the $d$-dimensional Poincar\'{e} ball, a canonical model for negatively curved spaces and hyperbolic geometry (\citep{cannon:floyd:kenyon:parry:1997,Ratcliffe2006}). The random geometric graph on the Euclidean space was introduced in \citep{Gilbert1961} as a model of radio communications and since then it has been a thriving research topic in probability, statistical physics and wireless networks (see \citep{Meester1996,penrose:2003,Yukich2006,Baccelli2009,Baccelli2010,Haenggi2012}). In recent times, the study of random geometric graphs has formed the base for study of random geometric complexes and its applications to topological data analysis (see \citep{bobrowski:kahle:2014}). In its simplest form, the random geometric graph can be constructed by taking a random set of iid points $\X_n = \{X_1,\ldots,X_n\}$ on a metric space as its vertex set and placing an edge between any two distinct points within a distance $r_n$. As is to be expected, most studies of such graphs assume that the underlying metric space is Euclidean or some compact, convex Euclidean subset. But various applications, especially the newish ones in topological data analysis, necessitate studies of geometric and topological structures on $\X_n$ with more general underlying metric spaces. Such extensions to compact manifolds without a boundary have been investigated recently in \citep{bobrowski:mukherjee:2015,penrose2013}, and a crude one-line summary of these studies is that the behaviour of the graph on a ``nice'' $d$-dimensional manifold is similar to that on a $d$-dimensional Euclidean space, though the proofs and the precise mathematical assumptions are quite challenging. \\

Given that the class of ``nice'' $d$-dimensional manifolds as considered above includes $d$-dimensional spheres (having constant positive curvature) and $d$-dimensional Euclidean spaces (having zero curvature), it is natural to ask about random geometric graphs on negatively curved spaces. Such an investigation was initiated recently in \citep{krioukov:papadopoulos:kitsak:vahdat:boguna:2010} on hyperbolic spaces and even more recently in \citep{cunningham2017} on more general spaces such as Lorentzian manifolds. However, since the $d$-dimensional Poincar\'{e} ball is one of the canonical and well-understood models of non-Euclidean and non-compact spaces, we shall restrict our attention to the same. Apart from the mathematical curiosity to understand random geometric graphs on negatively curved spaces, another reason to investigate hyperbolic random graphs arise from them being good models of many complex networks  exhibiting sparsity, power-law degree distribution, small-world phenomena, and clustering. For more details, see the introductions in \citep{krioukov:papadopoulos:kitsak:vahdat:boguna:2010,gugelmann:panagiotou:peter:2012,fountoulakis2015geometrization}. This graph is sometimes also referred to as the disc model or the KPKVB model after the authors of \citep{krioukov:papadopoulos:kitsak:vahdat:boguna:2010}, but we shall use the term hyperbolic random geometric graph. Though our work is a natural successor to this literature on hyperbolic random geometric graphs, our work can be considered, in a broader sense, as an addition to the developing literature about random structures on hyperbolic spaces (see also \citep{brooks2004,Benjamini2011,Benjamini2013,Lyons2016, Lalley2014,Petri2016}).  \\ 

The rest of the article is organized as follows : In the following subsections - Sections \ref{sec:hgg} and \ref{sec:resultsample} - we informally introduce the hyperbolic random geometric graphs, present some heuristics based on simulations, give a preview of our results and also discuss the background literature. Then, in Section \ref{sec:setup}, we introduce our setup in detail, mention some basic lemmas and state all our results. This is followed by the proofs in Section \ref{sec:proof}, where we also introduce basic lemmas on the hyperbolic metric, the hyperbolic measures as well as an abstract normal approximation bound in \citep{last:peccati:schulte:2016} derived from Malliavin-Stein method. Finally, in Section \ref{sec:Appendix}, we conclude with appendices on Palm theory for the Poisson point process and comparison with Euclidean random geometric graphs. 

\subsection{Hyperbolic random graphs:} 
\label{sec:hgg}
We shall quickly introduce the Poincar\'{e} ball and the hyperbolic random geometric graphs to give a preview of our results. Though there are other models of hyperbolic spaces, they are all isometric to the Poincar\'{e} ball (\citep[Section 7]{cannon:floyd:kenyon:parry:1997}). The Poincar\'{e} $d$-ball $\Bdzeta$ with negative curvature $-\zeta^2$ is the $d$-dimensional open unit ball equipped with the Riemannian metric 
\begin{equation} \label{e:hyp.metric}
ds^2 := \frac{4}{\zeta^2} \frac{|dx|^2}{(1-|x|^2)^2} = \frac{4}{\zeta^2} \frac{dx_1^2 + \cdots + dx_d^2}{(1-x_1^2 - \cdots - x_d^2)^2} , 
\end{equation}
where $|\cdot|$ denotes the Euclidean norm. We shall denote the metric by $d := d_{\zeta}$. See Section \ref{sec:setup} for more detailed definitions. We shall use $d$ to denote both the hyperbolic metric, and the dimension of an underlying space, but the context can distinguish the two sufficiently. Though $B_d^{(\zeta)}$ is topologically the same as any open Euclidean ball, what matters to us is the metric, and this is different from that of the Euclidean one. On compact sets of the unit ball, the hyperbolic metric is equivalent to the Euclidean metric, and the differences surface only as $|x| \uparrow 1$ but in a very significant way. To get an idea of the differences with the Euclidean space, see lines and circles on the Poincar\'e disk in Figure \ref{fig:poincare}. As is evidently expected, the unit line segments near the boundary look much smaller than those closer to the center, and line segments near the center are closer to straight lines, while those near the boundary are curved. While circles are always circles, the centers of the circles closer to the boundary are far away from the respective Euclidean centers. 
\begin{figure}[!htbp]
\centering
\includegraphics[width=3in,height=3in]{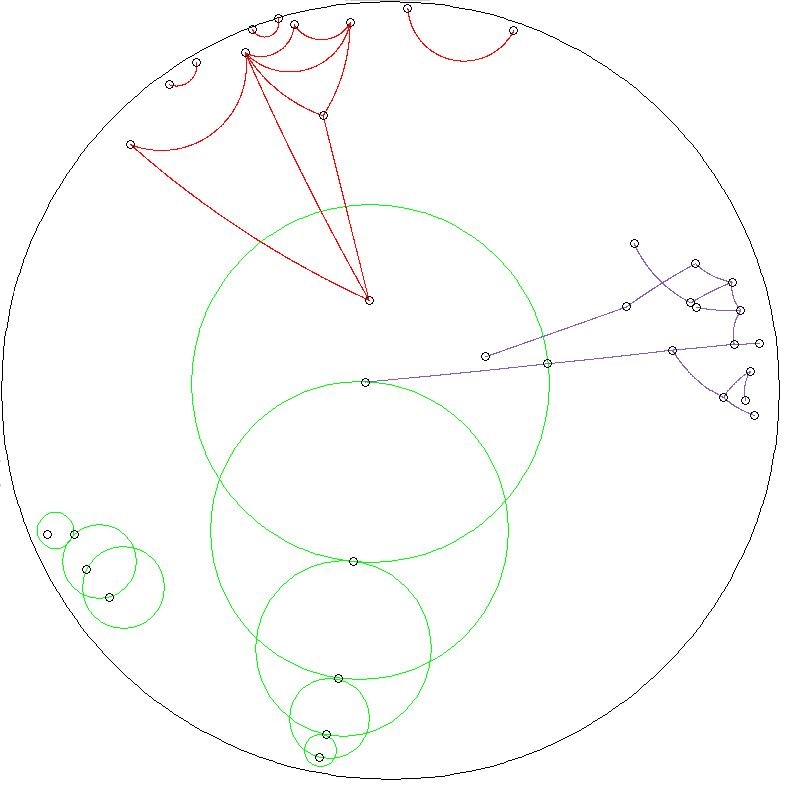} 
\caption{Geodesic line segments and triangles between points (red lines), geodesic line segments of unit length (violet lines), and unit circles with centers on the Poincar\'e disk (green circles) with $\zeta = 1$. These figures are drawn using the applet Noneuclid \citep{Noneuclid}.}
\label{fig:poincare}
\end{figure} 

Let $B(0,R)$ denote the hyperbolic ball of radius $R$ centred at the origin. In particular, if $d = 2$, the area of $B(0,R)$ is $2\pi \bigl( \cosh(\zeta R) - 1 \bigr)/\zeta^2$. Even in a higher-dimensional case, the volume of $B(0,R)$ grows exponentially in terms of the radius, and this is yet another aspect of hyperbolic spaces. Apart from the curvature parameter $-\zeta^2$, our hyperbolic random geometric graph shall involve a second curvature parameter $-\alpha^2$ of another Poincar\'{e} $d$-ball $\Bdalpha$. We shall choose a sequence of radii $R_n \to \infty$ as $n \to \infty$ and select $N_n \stackrel{d}{=}$ Poisson$(n)$ iid ``uniform" points $X_1,\ldots,X_{N_n}$ in $B(0,R_n) \subset \Bdalpha$ and project them onto $\Bdzeta$ preserving their polar coordinates. Then, we connect any two points $X_i,X_j$ if $0 < d_{\zeta}(X_i,X_j) < R_n$, i.e., we sample points uniformly in growing balls of $\Bdalpha$ and form the random geometric graph on $\Bdzeta$. We denote this random geometric graph by $HG_n(R_n ; \alpha,\zeta)$, for which there are four parameters involved : the dimension $d$, two curvature paramaters $\alpha$ and $\zeta$, and the radii regime $R_n$. We remark here that if we assume $R_n$ to be bounded by $R < \infty$, then $\Bdzeta \cap B(0,R)$ is metrically equivalent to a compact Euclidean ball, and thus, 
the asymptotics for such $HG_n(R_n ; \alpha,\zeta)$ will be very much the same as that of Euclidean random geometric graphs. For asymptotics of Euclidean random geometric graphs, see Section \ref{sec:comparison}. To illustrate the hyperbolic random geometric graph, we present five simulations for $d = 2$ in Figure \ref{fig:hgg} for different choices of $\alpha$ but with $n = 1000, \zeta = 1, R_n = 2\log 1000 = 13.82$. See also Figure \ref{fig:ergg} for simulations of two analogous Euclidean random geometric graphs.

There are few things about these figures we wish to point out.  Though $\zeta,\alpha$ are two parameters, we have fixed $\zeta = 1$ and varied $\alpha$ in our simulations. The reason for doing so is that the ratio $\alpha / \zeta$ is what matters and this will be obvious in the next subsection. It is useful to keep in mind that for small $\alpha$, the space behaves more like Euclidean in the sense that there are more points near the center which affect the asymptotics, whereas for large $\alpha$, the points near the boundary alone dominate the asymptotics. Further, by the geometry of the hyperbolic spaces, points near the center can connect easily to all the points, and so, the presence of such points changes the connectivity structure of graphs. 

One of the main characteristics of hyperbolic geometric graphs on the Poincar\'e ball is the presence of \textit{tree-like structures}, implying that the vertices on $B_d^{(\zeta)}$ are classified into large groups of smaller subgroups, which themselves consist of further smaller subgroups (see \citep{krioukov:papadopoulos:kitsak:vahdat:boguna:2010}).  This is reflected in our simulations,  indicating that there seem to be more sub-trees embedded than their Euclidean counterparts. To uncover the spatial distribution of such tree-like structures, this article will focus on sub-tree counts in the hyperbolic random geometric graph, i.e., the number of copies of a given tree $\Gamma_k$ of $k$ vertices in $HG_n(R_n ; \alpha, \zeta)$. In usual graph-theoretic language, we count the number of graph homomorphisms from $\Gamma_k$ to $HG_n(R_n ; \alpha, \zeta)$. We may notice a phase transition in the connectivity of the graph at $\alpha = 1$ due to appearance of points closer to the center. Some of the above observations that have been crystallized into rigorous mathematical theorems shall be mentioned in the next subsection, but many more still await to be explored. 
\begin{figure}[!htbp]
\centering
\caption{Simulations of $HG_{1000}(2\log 1000 = 13.82 ; \alpha, 1)$ for $d = 2$ with different $\alpha$. Isolated vertices have been omitted.}
\includegraphics[width=2.4in,height=2.4in]{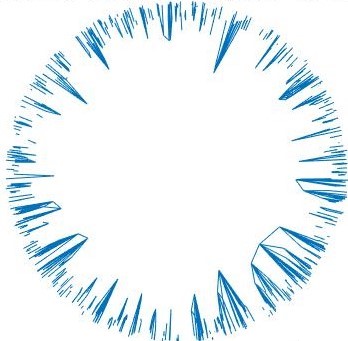} \hspace*{2.5cm}
\includegraphics[width=2.4in,height=2.3in]{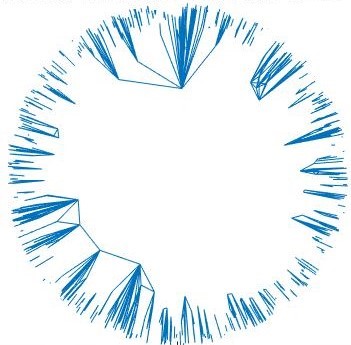} 
\caption*{$\alpha = 1.2$ \hspace*{6.9cm} $\alpha = 1.05$.}
\includegraphics[width=2.3in,height=2.3in]{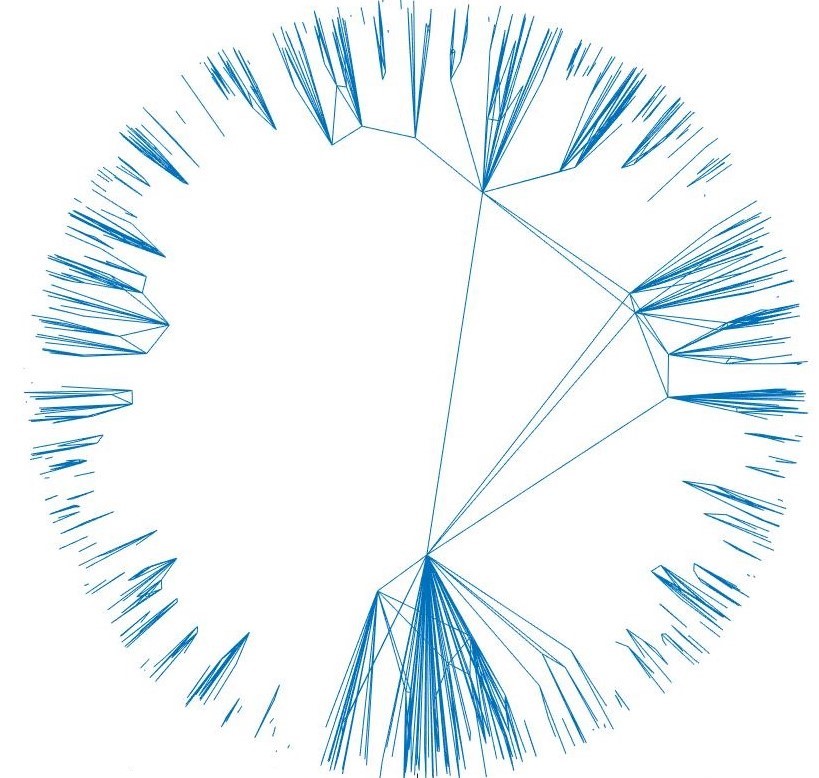} \hspace*{2.5cm}
\includegraphics[width=2.3in,height=2.3in]{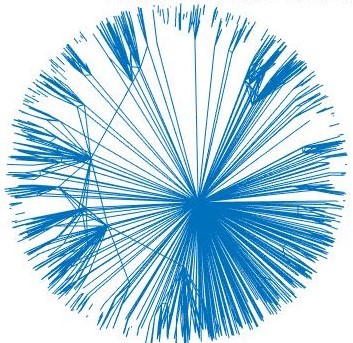} 
\caption*{$\alpha = 1$ \hspace*{6.5cm} $\alpha = .95$.}
\includegraphics[width=2.3in,height=2.3in]{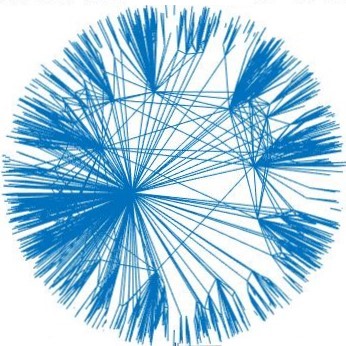} 
\caption*{$\alpha = 0.8$}
\label{fig:hgg}
\end{figure}
\subsection{A preview of our results:}
\label{sec:resultsample}
Earlier works on hyperbolic random geometric graphs (see \citep{abdullah2015,Bode2016,fountoulakis2015geometrization, candellero2016, fountoulakis2016law, fountoulakis:2012, Muller2017}) are mostly concerned with the case 
\begin{equation} 
\label{e.regime_literature}
d = 2 \ \  \text{ and } \ R_n = 2\zeta^{-1} \log (n/\nu) \ \  \text{for some }\nu \in  (0,\infty).
\end{equation}
However, many papers also consider the more general {\em binomial model}, where the probability of an edge between vertices $u,v$ is given by $(1 + \exp\{\beta \zeta (d_{\zeta}(u,v)-R_n)/2\})^{-1}$ for $\beta \in(0,\infty]$. The hyperbolic random geometric graph is a special case of the binomial model when $\beta = \infty$. Further, the binomial model under the regime \eqref{e.regime_literature} considered in the literature has been shown to be asymptotically a special case of geometric inhomogeneous random graphs (see \citep[Theorem 7.3]{Bringmann2017}). But our results cover more general radius regimes as well as apply in higher dimensions. This makes it difficult to use the existing results in \citep{Bringmann2017} about geometric inhomogeneous random graphs. In addition, in contrast to most of the existing literature, we prove second order asymptotics, i.e., variance asymptotics and central limit theorem. Though the present paper allows for more general choices of radius regime, this subsection shall restrict itself to the higher-dimensional version of \eqref{e.regime_literature}, namely
$$ R_n = 2[\zeta (d-1)]^{-1}\log (n/\nu), \ \ \nu \in (0,\infty), \ \ d \geq 2,$$ 
for the sake of an easier  presentation of our results. As for the corresponding hyperbolic random geometric graphs, we shall see that the expected average degree converges to a constant if $\alpha / \zeta > 1/2$ regardless of the dimension $d$. Such a regime can be referred to as {\em the thermodynamic regime}. We shall also find that, even in higher dimensions, the asymptotic behaviour of the hyperbolic random geometric graph is mainly determined by the ratio $\alpha/\zeta$. 

Let $\Gamma_k$ denote a tree on $k$ vertices ($k \geq 2$) with the ordered degree sequence $d_{(1)} \leq d_{(2)} \leq \ldots \leq d_{(k)}$. Our interest lies in the statistic $S_n$ which counts the number of subgraphs (not necessarily induced) in $HG_n(R_n ; \alpha, \zeta)$ isomorphic to $\Gamma_k$. We call $S_n$  {\em sub-tree counts}. Our most general results are for sub-tree counts $\Tngamma$  on $B(0,R_n) \setminus \mathring{B}(0, (1-\gamma)R_n)$ ( $\mathring{}$ denoting the interior of the set) for $\gamma \in (0,1)$. Many of our proofs proceed by deriving asymptotics for $\Tngamma$ and then approximating $S_n$ (which is nothing but $S_n^{(1)}$) by $\Tngamma$ for small enough $\gamma$. Such a strategy is very much due to the behaviour of the Poincar\'e  ball near its boundary. We assume $\alpha /\zeta$ is not a natural number for simplifying the statements of our results. An important consequence of expectation and variance asymptotics may be summarised very quickly as follows (more details are given in Section \ref{sec:specialcase}) : There exist explicit constants $C_1 := C_1(\zeta,\alpha,d,\nu,\Gamma_k), C_2 := C_2(\zeta,\alpha,d,\nu,\Gamma_k)$ such that for $\gamma \in (0,1) \setminus \{1/2\}$, \footnote{Here $a_n \sim b_n$ denotes that $a_n / b_n \to 1$ and further we use the standard Bachman-Landau big-O little-o notation.}
\begin{align*} 
 \E(\Tngamma) &\sim  C_1 n^{1 + \gamma\sum_{i=1}^k (d_{(i)} - 2\alpha/ \zeta)_+}, \\
 \V(\Tngamma) &= \Omega\Bigl( n^{1 + 2\gamma(d_{(k)} - \alpha/ \zeta)_+ + 2\gamma \sum_{i=1}^{k-1} (d_{(i)} - 2\alpha/ \zeta)_+} \, \vee \, n^{1 +\gamma \sum_{i=1}^k (d_{(i)} - 2\alpha/ \zeta)_+ } \Bigr),
\end{align*}
where $a \vee b = \max \{  a,b\}$ for $a,b\in\bbr$, and $(a)_+ = a$ if $a>0$ and $(a)_+=0$ otherwise. 
Further, for $\gamma$ small enough, the central limit theorem (CLT) holds for $\Tngamma$ as well. As for $\gamma = 1/2$, we have that
$$  \E(\Tngamma) = \Theta(n^{1 + 2^{-1}\sum_{i=1}^k (d_{(i)} -  2\alpha/ \zeta)_+}).$$
The above result (i.e., the $\gamma = 1/2$ case) for $d = 2$ was shown in \citep[Claim 5.2]{candellero:fountoulakis:2016}. \\
Now, let $\gamma \in (0,1]$ . If $2 \alpha / \zeta > d_{(k)} $, we have that as $n\to\infty$, 
\begin{align*}
\E (S_n) \sim \E(\Tngamma) \sim  C_1 n. 
\end{align*}
If $\alpha / \zeta > d_{(k)}$, then as $n \to \infty$, 
$$\V(S_n) \sim \V (\Tngamma) \sim  C_2n,$$
and also, the CLT holds for $S_n$. As a comparison, for Euclidean random geometric graphs in the thermodynamic regime, we have that $\E (S_n)  = \Theta(n), \V(S_n) = \Theta(n)$, and the central limit theorem holds as well (See Section \ref{sec:comparison}). The heuristic explanation is that a larger $\alpha / \zeta$ ratio means that the space $\Bdalpha$ is more hyperbolic relative to $\Bdzeta$ and hence contains more points in the boundary that dominate the contribution to $S_n$. In other words, $\Tngamma$ dominates the contribution to $S_n$. 

Again, if we choose $k =2$, then $S_n$ is nothing but the number of edges and $\E(S_n/2n)$ is the expected average degree.  As we see from the above expectation asymptotics, the expected average degree is $\Theta(1)$ (i.e., thermodynamic regime) if $2\alpha /\zeta > 1$. This is one of the reasons for an assumption like $\alpha /\zeta > 1/2$ in many of the earlier papers. For  $2\alpha/ \zeta > 1$, the convergence of the expected average degree to a constant  is consistent with the power-law behaviour of degree distribution with exponent $2\alpha /\zeta +1$, which itself was predicted in \citep{krioukov:papadopoulos:kitsak:vahdat:boguna:2010}. Such a power-law behaviour for degree distribution has been proven in \citep{gugelmann:panagiotou:peter:2012,fountoulakis:2012} for $d = 2$. For  $2\alpha /\zeta \leq 1$, the expected average degree grows to infinity, which is again consistent with the conjecture that the degree distribution has a power-law behaviour with exponent $2$ (\citep{krioukov:papadopoulos:kitsak:vahdat:boguna:2010}).

An interesting consequence of the behaviour controlled by $\alpha /\zeta$ is that the asymptotics for uniformly distributed Poisson points on any $\Bdzeta$ with $\alpha = \zeta$ are unaffected by changes in $\zeta$ or dimension. For example, the expected number of edges grows linearly for all $\Bdzeta, d \geq 2$ with $\alpha = \zeta$, but as for the variance, we only have the lower bound $\Omega(n)$. 

For further discussion on related results, we first refer the reader to the following table summarising some of the existing literature and our results in the special case $d = 2, \zeta = 1, R_n = 2 \log (n / \nu), \nu \in (0,\infty)$. 
\begin{table}[h]
\large
 \begin{minipage}{\textwidth}
\centering 
\begin{tabular}{|c|c|c|c|c|}
\hline
Regime / Properties & $\P_{conn}$ & $\P_{perc}$ & $\E(K_k), k \geq 2$ & $\E(\Tngamma), \gamma \in (0,1) \setminus \{1/2\}$ \\
\hline
Results from & \citep{Bode2016} & \citep{bode:fountoulakis:muller:2015} & \citep{friedrich:krohmer:2015} , \citep{candellero:fountoulakis:2016} & Corollary \ref{cor:specialcase} \\
\hline
$\alpha < 1/2$ &  $1$ & $1$ & Not Known &  $\sim n^{1 + 2\gamma (k- 1 - k \alpha)}$ \\
\hline 
$1/2 < \alpha \leq 1 - 1/k$ & $0$ & $1$ & $\Theta(n^{(1-\alpha)k})$ & $\sim n^{1 + \gamma\sum_{i=1}^k (d_{(i)} - 2\alpha)_+}$ \\
\hline 
$1 - 1/k < \alpha < 1$ & $0$ & $1$ & $\Theta(n)$ & $\sim n^{1 + \gamma\sum_{i=1}^k (d_{(i)} - 2\alpha )_+}$ \\
\hline 
$1 < \alpha < d_{(k)}/2$ & $0$ & $0$ & $\Theta(n)$ & $\sim n^{1 + \gamma\sum_{i=1}^k (d_{(i)} - 2\alpha )_+}$ \\
\hline 
$d_{(k)}/2 < \alpha$ & $0$ & $0$ & $\Theta(n)$  & $\sim n$ \\
\hline
\end{tabular}
\caption{Summary of related results for $d = 2, \zeta = 1, R_n = 2 \log (n/\nu), \nu \in (0,\infty)$. Here $K_k$ denotes the number of $k$-cliques in $HG_n(R_n ; \alpha, 1)$, $\P_{conn} = \P(\mbox{ $HG_n(R_n ; \alpha, 1)$ is connected})$ and $\P_{perc} = \P(\mbox{$HG_n(R_n ; \alpha, 1)$ percolates})$, where, by percolation, we mean existence of a giant component, i.e., a component of size $\Theta(n)$. }
\label{table:summary}
\end{minipage}
\end{table} 

In comparison to other results, our results demonstrate a completely different phase transition for sub-tree counts  in the sense that it depends not just upon the size of the trees but also the degree sequence. Thus, for $1 <  \alpha / \zeta < d_{(k)}/2$, we have that sub-tree counts grow super-linear in $n$, even though $\E(K_k)$ is linear, and there is no ``giant component". Such a phenomenon is further evidence of our observation based on simulations that the hyperbolic random geometric graphs contain many ``tree-like" structures compared to its Euclidean counterpart (see Figure \ref{fig:ergg} and Section \ref{sec:comparison}). A more mathematical reason for hyperbolic random geomtric graph supporting tree-like structures is that non-amenability of negatively curved spaces are more conducive to embedding of trees compared to Euclidean spaces. 

A few words on our proofs. The expectation and variance asymptotics for $\Tngamma$ involve Palm theory for Poisson point process and various estimates for the measure and the metric on the Poincar\'e ball. The need for the tree assumption arises because the hyperbolic metric involves relative angles between points, and the relative angles in the tree-like structure exhibit sufficient independence (see Lemma \ref{l:int_angle_tree}) for our precise calculations. For the central limit theorem, we use the abstract normal approximation result (\citep{last:peccati:schulte:2016}) derived using the Malliavin-Stein method. To use this normal approximation result, we derive detailed bounds on the first order (add-one cost) and second order difference operators of the functional $\Tngamma$. As mentioned before, extending results from $\Tngamma$ to $S_n$ always involves showing that the more hyperbolic the space is, the boundary contributions dominate those arising from near the center.

We shall end the introduction with a few pointers about the wider literature on hyperbolic random geometric graphs. For more on percolation and connectivity, refer to \citep{Bode2016,candellero2016,fountoulakis2016law,Kiwi2017} and studies on typical distances and diameter can be found in \citep{abdullah2015,Muller2017,kiwi2014bound}. Spectral properties of these graphs are studied in \citep{Kiwi2016spectral}. An interesting aspect apart from those mentioned above is the similarity of this random graph model to the Chung-Lu inhomogeneous random graph model (\citep{fountoulakis2015geometrization}), and this has been exploited in \citep{fountoulakis2016law,candellero2016}.  We leave generalization of our results to the Binomial model and the geometric inhomogeneous random graphs for future work. 
\section{{\bf Our setup and results}}
\label{sec:setup}
\subsection{The Poincar\'e ball}
\label{sec:poincare}~\\
Our underlying metric space is the  \textit{$d$-dimensional Poincar\'e ball} $B_d^{(\zeta)}$, where $-\zeta^2$ represents the negative (Gaussian) curvature of the space with $\zeta > 0$, i.e.,
$$
B_d^{(\zeta)} := \bigl\{  (x_1,\dots,x_d) \in \bbr^d: x_1^2+ \cdots + x_d^2 < 1\bigr\}
$$
and is equipped with the Riemannian metric \eqref{e:hyp.metric}. We shall now mention some basic properties of this metric space, and some more properties will be stated in Section \ref{sec:hyperbolicapp}. For more details on the Poincar\'e ball, we refer the reader to \citep{Ratcliffe2006,anderson:2008} and for a quick reading, refer to \citep{cannon:floyd:kenyon:parry:1997}. 
%
In what follows, we often represent the point $x \in B_d^{(\zeta)}$ in terms of ``hyperbolic" polar coordinate; For $x \in B_d^{(\zeta)}$, we write $x = (r,\theta_1,\dots, \theta_{d-1})$, where $r\geq 0$ is the radial part of $x$ defined by 
\begin{equation*}
r = \frac{1}{\zeta} \log \frac{ 1+|x|}{ 1-|x|}
\end{equation*}
and $(\theta_1, \dots, \ta_{d-1}) \in C_d := [0,\pi]^{d-2} \times [0,2\pi)$ is the angular part of $x$. Let $d:=d_\zeta$ denote the hyperbolic distance induced by \eqref{e:hyp.metric}, then it satisfies $d(0,x)= r$. 

Using the hyperbolic polar coordinate $(r,\ta_1,\dots,\ta_{d-1})$, the metric \eqref{e:hyp.metric} can be rewritten as 
\begin{equation}
\label{e:hmetric}
ds^2 = dr^2 + \Bigl( \frac{\sinh (\zeta r)}{\zeta} \Bigr)^2 \bigl( d\theta_1^2 + \sum_{k=2}^{d-1} \prod_{i=1}^{k-1} \sin^2 \theta_i \, d\theta_k^2 \bigr), 
\end{equation}
from which we obtain the volume element :
$$
dV = \left( \frac{\sinh (\zeta r)}{\zeta} \right)^{d-1}\prod_{i=1}^{d-2} \sin^{d-i-1} \ta_i\, dr\, d\ta_1 \dots d\ta_{d-1}. 
$$
%
%
%
%
%
We now aim to  generate random points on a sequence of growing compact subsets of the Poincar\'e ball. First, we choose a deterministic sequence $R_n$, $n\geq1$, which grows to infinity as $n\to\infty$. We assume that the angular part of random points is uniformly chosen, i.e., the probability density is
\begin{equation} \label{e:angular.pdf}
\pi (\ta_1, \dots, \ta_{d-1}) = \frac{\prod_{i=1}^{d-2} \sin^{d-i-1} \theta_i}{2\prod_{i=1}^{d-1}\kappa_{d-i-1}}, \ \ \ (\ta_1,\dots,\ta_{d-1}) \in C_d,
\end{equation}
where we have set $\kappa_m= \int_0^{\pi} \sin^m \theta\, d\theta$, and trivially, $\kappa_0 = \pi$. We use the symbol $\pi$ to denote the angular density as well as the famed constant, but the context makes it clear which of them we refer to. Given another parameter $\alpha > 0$, we assume that the density of the radial part is 
\begin{equation} \label{e:radial.pdf}
\rhona(r) = \frac{\sinh^{d-1}(\alpha r)}{\int_0^{R_n} \sinh^{d-1} (\alpha s) ds}, \ \ 0 \leq r \leq R_n. 
\end{equation}
This density is described as the ratio of the surface area of $B(0,r)$ to the volume of $B(0,R_n)$, where $B(0,r)$ and $B(0,R_n)$ are both defined on $B_d^{(\alpha)}$. So, the density \eqref{e:radial.pdf} can be regarded as a uniform density (for the radial part) on $B_d^{(\alpha)}$.  Combining \eqref{e:angular.pdf} and \eqref{e:radial.pdf} together, we can generate uniform random points on the space $B(0,R_n) \subset \Bdalpha$, and then, we project all of these points onto the original Poincar\'e ball $\Bdzeta$, where we construct the hyperbolic geometric graph. The projection is such that the polar coordinates remain the same. Obviously,  \eqref{e:radial.pdf} is no longer a uniform density (for the radial part) on $\Bdzeta$, unless $\zeta = \alpha$. 

Though the probability density in \eqref{e:radial.pdf} looks a little complicated, we shall mostly resort to the following useful approximation via a suitable exponential density. Set $T:= R_n-d(0,X)$, where $X$ is a random variable with density $\rhona \times \pi$. Denote by $\barrhona (t)$, the density of $T$, i.e., 
\begin{equation}
\label{e:barrho}
\barrhona (t) = \frac{\sinh^{d-1}\bigl(\alpha (R_n-t)\bigr)}{\int_0^{R_n} \sinh^{d-1} (\alpha s) ds}, \ \ 0 \leq t \leq R_n. 
\end{equation}
In the sequel, we often denote a random variable $X$ with density $\rho_{n,\alpha} \times \pi$ by its hyperbolic polar coordinate $X=(T, \Ta)$, where $\Ta$ is  the angular part, and the radial part is described by $T$ rather than $d(0,X)$.

The approximation result below was established for $d=2$ in \citep{candellero:fountoulakis:2016}. The formal proof is given in Section \ref{sec:proof}. 
\begin{lemma} \label{l:pdf}
(i) As $n\to\infty$, we have 
$$
\barrhona (t) \leq \bigl( 1+o(1) \bigr) \alpha (d-1) e^{-\alpha (d-1)t}
$$
uniformly for $0 \leq t < R_n$. 

\noindent (ii) For every $0 < \lambda <1$, we have, as $n\to \infty$, 
$$
\barrhona (t) =\bigl( 1+o(1) \bigr) \alpha (d-1) e^{-\alpha (d-1)t}
$$ 
uniformly for $0 \leq t \leq \lambda R_n$. 
\end{lemma}

Now we define our main object of interest, the hyperbolic random geometric graph. The first ingredient is the Poisson point process on $\Bdzeta$. For every $n \geq1$, let $(X_i, \, i \geq 1)$ be a sequence of iid random points on $\Bdzeta$ with common density $\rhona \times \pi$. Letting $N_n$ be a Poisson random variable with mean $n$, independent of $(X_i)$, one can construct the Poisson point process $\cP_n  = \{  X_1, X_2, \dots, X_{N_n}\}$ whose intensity measure is $n (\rhona \times \pi)$. 
\begin{definition}[Hyperbolic random geometric graph]
\label{defn:HGG}
Let $R_n$ be a sequence growing to $\infty$ as $n \to \infty$. The hyperbolic random geometric graph $HG_n(R_n ; \alpha, \zeta)$ is a simple, undirected graph whose vertex set is the Poisson point process $\cP_n$ with intensity measure $n (\rhona \times \pi)$, and the edge set is $\{ (X_i,X_j) : X_i,X_j \in \cP_n, 0 < d(X_i,X_j) \leq R_n \}$, where $d$ is the hyperbolic metric on $\Bdzeta$ induced by \eqref{e:hyp.metric} or equivalently \eqref{e:hmetric}. \\
\end{definition}
%


\subsection{Expectation and variance asymptotics for sub-tree counts}~\\
\label{sec:exp_var}
We introduce the notion of graph homomorphisms to define sub-tree counts. Suppose $H$ is a simple graph on $[k]:=\{ 1,\dots,k \}$ with edge set $E(H)$. Given another simple graph $G = (V(G),E(G))$, a {\em graph homomorphism} from $H$ to $G$ refers to a function $f : [k] \to V(G)$ such that if $(i,j) \in E(H)$, then $(f(i),f(j)) \in E(G)$, i.e., the adjacency relation is preserved. We denote by $\C (H,G)$ the number of graph homomorphisms from $H$ to $G$, that is, the number of copies of $H$ in $G$. This can be represented easily as follows :
\begin{equation}
\label{e:homrep}
 \C(H,G) = \sum_{(v_1,\ldots,v_k) \in V(G)}^{\neq} \prod_{(i,j) \in E(H)} \one \bigl\{ \, (v_i,v_j) \in E(G)\, \bigr\},
\end{equation}
where $\sum^{\neq}$ denotes the sum over distinct $k$-tuples $v_1,\ldots,v_k$ and $\one \{  \cdot \}$ is an indicator function. It is worth mentioning that the subgraphs counted by $\C(H,G)$ are not necessarily induced subgraphs in $G$ isomorphic to $H$. From the above representation, it is easy to derive the following monotonicity property : If $H_1,H_2$ are simple graphs on $[k]$ such that $E(H_1) \subset E(H_2)$, then $\C(H_2,G) \leq \C(H_1,G)$. 
 
Let us return to our setup of hyperbolic random geometric graphs as in Definition \ref{defn:HGG}. If $N_n \geq k$, we denote a collection of $k$-tuples of distinct elements in $\Pn$ by 
\begin{equation}
\label{e:distinctpts}
\mathcal P_{n,\neq}^k := \bigl\{ (X_{i_1},\dots,X_{i_k}) \in \Pn^k:  i_j \in \{1,\ldots,N_n\}, \ i_j \neq i_\ell \ \mbox{for} \ j \neq \ell \bigr\}. 
\end{equation}
Set $\mathcal P_{n,\neq}^k = \emptyset$ if $N_n < k$. 
Define the annulus $D_\gamma(R_n) := B(0,R_n) \setminus \mathring{B}(0, (1-\gamma)R_n)$ for $0 < \gamma \leq 1$. Construct the hyperbolic random geometric graph on $\Pn \cap D_\gamma(R_n)$ as in Definition \ref{defn:HGG}, and we denote it as $HG_n^{(\gamma)}(R_n ; \alpha, \zeta)$.  As in Section \ref{sec:resultsample}, we set $\Gamma_k$ to be a tree on $[k]$ with edge set $E$. We exclude the trivial choice of $k = 1$, in which case, $\Gamma_k$ represents a single vertex. Our interest is in sub-tree counts $\Tngamma := \C\bigl(\Gamma_k,HG^{(\gamma)}_n(R_n; \alpha, \zeta)\bigr)$ for $\gamma \in (0,1]$. From \eqref{e:homrep}, by denoting $T_i = R_n-d(0,X_i)$, we have that sub-tree counts $\Tngamma$ can be represented as
\begin{equation}
\label{e:tngamma}
\Tngamma = \sumXkPn \prod_{(i,j) \in E} \one \bigl\{ 0<d(X_i, X_j) \leq R_n,  \ T_i, T_j \leq \gamma R_n, \}. 
\end{equation}
In particular, we write $S_n = S_n^{(1)}$. Obviously, we have $S_n = \C\bigl(\Gamma_k,HG_n(R_n; \alpha, \zeta)\bigr)$. Our  first result gives the asymptotic growth rate of $\E(\Tngamma)$ for $\gamma \in (0,1]$. 
\begin{theorem} \label{t:expectation.Tngamma}
Let $\Gamma_k$ be a tree on $k$ vertices ($k \geq 2$) with degree sequence $d_1,\ldots,d_k$ and $\Tngamma$ be the sub-tree counts as defined in \eqref{e:tngamma}.
For $ \gamma \in (0,1) \setminus \{1/2\}$, we have that as $n \to \infty$,
\begin{equation} \label{e:expectation.Tngamma}
\E \bigl( \Tngamma \bigr) \sim \Bigl( \frac{2^{d-1}}{\kappa_{d-2}}\Bigr)^{k-1} \alpha^k (d-1) \, n^k e^{-\zeta (d-1) (k-1) R_n/2} \prod_{i=1}^k \angamma(d_i),
\end{equation}
where 
\begin{equation*}
\angamma(p) := \int_0^{\gamma R_n} e^{\zeta(d-1)(p-2\alpha / \zeta)t/2} dt, \, \, p \in \mathbb{N}_+.
\end{equation*}
For $\gamma = 1/2$, we have that as $n \to \infty$,
\begin{equation}  \label{e:exp.gamma.one.half}
\E \bigl( \Tngamma \bigr) = \Theta(n^k e^{-\zeta (d-1) (k-1) R_n/2} \prod_{i=1}^k a_n^{(1/2)}(d_i)),
\end{equation}
Further, let $d_{(1)} \leq d_{(2)} \leq \ldots \leq d_{(k)}$ be the degree sequence of $\Gamma_k$ arranged in ascending order. If $2\alpha /\zeta > d_{(k)}$, then, for all $\gamma \in (0,1]$, we have that as $n\to\infty$,
\begin{align}
\E ( S_n ) & \sim \E (\Tngamma) \sim \biggl( \frac{2^{d-1}}{(d-1)\kappa_{d-2}}  \biggr)^{k-1} \alpha^k \, \prod_{i=1}^k \bigl( \alpha -\frac{\zeta d_i}{2}\bigr)^{-1} n^k e^{-\zeta (d-1) (k-1)R_n/2}. \label{e:expectation.Sn}  
\end{align}
\end{theorem}
This theorem indicates that the asymptotics of $\E ( \Tngamma )$ are crucially determined by  underlying curvatures $-\zeta^2$ and $-\alpha^2$. To make the implication of the theorem more transparent, we fix $\zeta$ and think of $\alpha$ as a parameter. If $\alpha >0$ is sufficiently large, i.e., the space $\Bdalpha$ is sufficiently hyperbolic, sub-trees are asymptotically dominated by the contributions near the boundary of $B(0,R_n)$. More specifically, if $2\alpha /\zeta > d_{(k)}$, then $\angamma (d_i)$, $i=1,\dots,k$ all converge to a positive constant, and thus, the growth rate of $\E ( \Tngamma )$ does not depend on $\gamma$, implying that the spatial distribution of sub-trees is completely determined by those near the boundary of $B(0,R_n)$. In fact, for each $\gamma \in (0,1)$, the growth rate of $\E ( \Tngamma )$ coincides with that of $\E( S_n )$. 

On the other hand, if $\alpha$ becomes smaller, i.e., the space $\Bdalpha$ becomes flatter, then the spatial distribution of sub-trees begins to be affected by those scattered away from the boundary of $B(0,R_n)$. For example, if $d_{(k-1)} < 2\alpha /\zeta < \Dk $, we see that as $n \to \infty$
$$
\angamma (\Dk) \sim \frac{1}{d-1}\, \Bigl( \frac{\zeta d_{(k)}}{2} - \alpha  \Bigr)^{-1} e^{\zeta(d-1) (\Dk-2\alpha /\zeta)\gamma R_n/2} \, \, \, (\to \infty)
$$
while $\angamma (\Di)$, $i=1,\dots,k-1$, all tend to a positive constant. In this case, the growth rate of $\E( \Tngamma )$ is no longer independent of $\gamma$ and its growth rate becomes faster as $\gamma \nearrow 1$, i.e., as the inner radius of the corresponding annulus shrinks. Moreover, if $\alpha$ becomes even smaller so that $d_{(k-2)}< 2\alpha/\zeta < d_{(k-1)}$, then $\angamma (d_{(k-1)})$ also asymptotically contributes to the growth rate of $\E ( \Tngamma )$. Ultimately, if $0 < 2\alpha /\zeta < \Done$, then all of the $\angamma (d_i)$'s contribute to the growth rate of $\E ( \Tngamma )$. 


The corollary below claims that if an underlying tree $\Gamma_k$ is of the simplest form, satisfying $d_{(k)} = k-1$, $d_{(k-1)} = \cdots =d_{(1)} =1$, which represents a tree of a single root and $k-1$ leaves, then even more can be said about the asymptotics of $\log \E(  S_n)$, regardless of the values of $\alpha/\zeta$. 

\begin{corollary} \label{cor:expectation.simple}
Let $\Gamma_k$ be the tree on $k$ vertices ($k \geq 2$) with degree sequence $d_{(1)}= d_{(2)} = \cdots =  d_{(k-1)} = 1$ and $d_{(k)} = k-1$. 
Moreover, assume that $R_n$ satisfies
$$
\frac{R_n}{log n} \to c, \ \ n\to \infty, \ \ \text{for some }  c\in [0,\infty]
$$ 
(note that $c=0$ or $\infty$ is possible). Let $a \vee b = \max \{a,b  \}$ for $a,b\in \bbr$. \\
$(i)$ If $2\alpha /\zeta > k-1$,
\begin{align}
\E (  S_n) \sim \Bigl( \frac{2^{d-1}}{(d-1)\kappa_{d-2}}\Bigr)^{k-1} & \alpha^k  \, \bigl( \alpha -\zeta(k-1)/2\bigr)^{-1} \label{e:special.case.ESn}  \\
&\times \bigl( \alpha - \zeta/2 \bigr)^{-(k-1)} n^k e^{- \zeta (d-1)(k-1)R_n/2}, \ \ \ n\to\infty. \notag
\end{align}
$(ii)$ If $1 < 2\alpha/\zeta \leq k-1$, 
$$
\frac{\log \E(S_n )}{R_n \vee \log n} \to k(c\vee 1)^{-1} - \alpha(d-1  ) (1\vee c^{-1})^{-1}, \ \ \ n\to\infty.
$$
$(iii)$ If $0 < 2\alpha /\zeta \leq 1$, 
$$
\frac{\log \E(S_n )}{R_n\vee\log n} \to k(c\vee 1)^{-1}-(d-1)\bigl( \alpha k-  \zeta(k-1)/2 \bigl) (1\vee c^{-1})^{-1}, \ \ \ n\to\infty.
$$
\end{corollary}

Having described the expectation asymptotics, our ultimate goal is to establish the CLT for the sub-tree counts $\Tngamma$ for $0 <\gamma \leq 1$. Before CLT, it is important to investigate the variance asymptotics. The theorem below provides an asymptotic lower bound for $\V(\Tngamma)$ up to a constant factor.  As expected, we shall see that the lower bound of $\V(\Tngamma)$ also depends on the ratio $\alpha /\zeta$. Similarly to Theorem \ref{t:expectation.Tngamma}, if $\alpha /\zeta > \Dk$, the lower bound of $\V(\Tngamma)$ is independent of $\gamma$, whereas it depends on $\gamma$ when $\alpha /\zeta \leq \Dk$. Furthermore, if $\alpha /\zeta > \Dk$, we are able to establish the exact growth rate of $\V(S_n)$. In what follows, $C^*$ denotes a generic positive constant, which may vary between lines and does not depend on $n$. 
\begin{theorem}  \label{t:variance.Tngamma}
Let $\Gamma_k$ be a tree on $k$ vertices ($k \geq 2$) with degree sequence $d_1,\ldots,d_k$ and $\Tngamma$ be the sub-tree counts as defined in \eqref{e:tngamma}. For $0 < \gamma <1$,  
\begin{align}
\V(\Tngamma) &= \Omega \biggl( n^{2k-1}e^{-\zeta(d-1)(k-1)R_n} \angamma(2d_{(k)})\,  \prod_{i=1}^{k-1} \angamma (d_{(i)})^2
\, \vee \, n^k e^{-\zeta(d-1) (k-1)R_n/2} \prod_{i=1}^k \angamma (d_i) \biggr). \label{e:variance.Tngamma} 
\end{align}
Suppose further that $\alpha /\zeta > \Dk$, then, for all $\gamma \in (0,1]$, we have that
\begin{equation}  \label{e:variance.Sn}
\V(S_n)\sim \V(\Tngamma) \sim  C^* \Bigl[ n^{2k-1}e^{-\zeta(d-1)(k-1)R_n} \vee  n^k e^{-\zeta (d-1) (k-1)R_n/2}  \Bigr], \ \ n\to\infty.
\end{equation}
\end{theorem}
The main point of \eqref{e:variance.Sn} is that the growth rate of $\V(S_n)$ is determined by how rapidly $R_n$ grows to infinity. To see this in more detail, we shall consider a special case for which $R_n = c \log n$ for some $c \in (0,\infty)$. Let $g(m) = n^m e^{- \zeta (d-1) (m-1)R_n/2}$, $m \in \bbn_+$. Now assuming $\alpha / \zeta > d_{(k)}$ and observing whether $g(m)$ is non-increasing or not depends on the value of $c$, we get the following variance growth rates :

\begin{equation*}
\V(S_n) \sim  
\begin{cases}  C^* n^{2k-1 - c\zeta(d-1) (k-1)} & \text{if} \, \, 0 < c < 2 \zeta^{-1}(d-1)^{-1} \\ 
C^* n & \text{if} \, \, c=2 \zeta^{-1} (d-1)^{-1}\\
C^* n^{k -  c\zeta (d-1)(k-1)/2} & \text{if} \, \, c > 2 \zeta^{-1}(d-1)^{-1}. 
\end{cases}
\end{equation*} \\

\subsection{Central limit theorem for sub-tree counts}~
\label{sec:clt}\\
Having derived variance bounds, we take the article to its natural conclusion by proving a central limit theorem for $S_n^{(\gamma)}$. As was evident in expectation and variance results in Theorems \ref{t:expectation.Tngamma} and \ref{t:variance.Tngamma}, the ratio $\alpha / \zeta$ will have a bearing on the matter. Before stating the normal approximation result, we need to define the two metrics - Wasserstein distance $d_W$ and Kolmogorov distance $d_K$ - to be used. Let $Y_1,Y_2$ be two random variables and $Lip(1)$ be the set of Lipschitz functions $h : \bbr \to \bbr$ with Lipschitz constant of at most $1$.  
\begin{eqnarray*}
d_W(Y_1,Y_2) & = & \sup_{h \in Lip(1)}\bigl|\, \E\bigl(h(Y_1)\bigr) - \E\bigl(h(Y_2)\bigr)\, \bigr|, \\
d_K(Y_1,Y_2)  & = & \sup_{x \in \bbr}\,  \bigl|\, \P(Y_1 \leq x) - \P(Y_2 \leq x)\, \bigr|. \nonumber
\end{eqnarray*} 
Even though we have defined $d_W, d_K$ as a distance between two random variables, they are actually a distance between two probability distributions. 
Let $N$ denote the standard normal random variable and $\Rightarrow$ denote weak convergence in $\bbr$. In our proof, we derive more detailed bounds, albeit complicated, that also indicate the scope for improvement. 
\begin{theorem}
\label{t:clt.Tngamma}
Let  $\Gamma_k$ be a tree on $k$ vertices ($k \geq 2$) with degree sequence $d_1,\ldots,d_k$ and $\Tngamma$ be the sub-tree counts as defined in \eqref{e:tngamma}. Assume further that $R_n$ satisfies $ne^{-\zeta(d-1) R_n/2} \to c \in (0,\infty]$. For every $0 < a < 1/2$, there exists $0 < \gamma_0 < 1/2$ such that for all $0 < \gamma < \gamma_0$, we have that
\begin{equation}
\label{e:dW.Tngamma}
d_W\left( \frac{\Tngamma - \E(\Tngamma)}{\sqrt{\V(\Tngamma)}},N\right) = O(n^{-a}), \, \, \, \, d_K\left( \frac{\Tngamma - \E(\Tngamma)}{\sqrt{\V(\Tngamma)}},N\right) = O(n^{-a}).
\end{equation} 
Further, if $\alpha /\zeta > \Dk$, then
\begin{equation}
\label{e:clt.tn}
\frac{S_n - \E(S_n)}{\sqrt{\V(S_n)}} \Rightarrow N, \ \ \ n\to\infty.
\end{equation}
\end{theorem}
%
%
%
%
\vspace*{0.5cm}
\subsection{Special case :  $R_n = 2[\zeta (d-1)]^{-1} \log (n/\nu), \nu \in (0,\infty)$}~
\label{sec:specialcase}

The objective of this short section is to restate our results in the special case of $R_n = 2[\zeta (d-1)]^{-1} \log (n/\nu)$ for a positive constant $\nu \in (0,\infty)$. For $\alpha / \zeta > 1/2$, this corresponds to the thermodynamic regime as the average degree is asymptotically constant. We state these results here so as to enable easier comparison with precedent studies. This is also exactly what we assumed in the Section \ref{sec:resultsample} and clearly weaker than our assumptions in Sections \ref{sec:exp_var} and \ref{sec:clt}.
Under this scheme, it is easy to see that as $n \to \infty$, 
\begin{align}
\angamma (p) &\sim \bigl| (d-1) (\alpha  - \zeta p /2) \bigr|^{-1} \Bigl( \frac{n}{\nu} \Bigr)^{\gamma( p-2\alpha/ \zeta )_+} \one \bigl\{ p \neq 2\alpha /\zeta  \bigr\}  + \frac{2\gamma}{\zeta(d-1)}\, \log \Bigl( \frac{n}{\nu} \Bigr)\, \one \bigl\{ p =  2\alpha/ \zeta  \bigr\}, \ \ \ p \in \bbn_+, \label{e:angamma.special} 
\end{align}
where $(a)_+ = a$ if $a>0$ and $(a)_+=0$ otherwise. 
The result below simplifies the situation by focusing on a further special case, for which $2\alpha /\zeta$ is not an integer (to drop the second line in \eqref{e:angamma.special}). 
\begin{corollary}
\label{cor:specialcase}
Let $R_n = 2[\zeta (d-1)]^{-1}  \log (n/\nu)$, and $\Gamma_k$ be a tree on $k$ vertices ($k\geq 2$) with degree sequence $d_1,\dots,d_k$ and $\Tngamma$ be a sub-tree counts as defined in \eqref{e:tngamma}. Suppose that $2\alpha /\zeta$ is not an integer. For $\gamma  \in (0,1) \setminus \{1/2\}$, we have, as $n\to\infty$, 
\begin{align*} 
\E(\Tngamma) \sim \biggl( \frac{2^{d-1}}{(d-1)\kappa_{d-2}}\biggr)^{k-1} \alpha^k  &\prod_{j=1}^k\,  \Bigl|\,  \alpha  - \frac{\zeta}{2} d_j \, \Bigr|^{-1} \nu^{k-1-\gamma \sum_{i=1}^k (d_i - 2\alpha /\zeta)_+}  n^{1 + \gamma\sum_{i=1}^k (d_i - 2\alpha /\zeta)_+},  
\end{align*}
and for $\gamma = 1/2$, we have that 
\begin{equation}
\label{eqn:claim5.2}
\E(\Tngamma) = \Theta(n^{1 + 2^{-1}\sum_{i=1}^k (d_i - 2\alpha /\zeta)_+}).
\end{equation}
Further, as for variance asymptotics, we have for $\gamma \in (0,1)$,
\begin{align*}
\V(\Tngamma) &= \Omega \Bigl( n^{1 + 2\gamma(d_{(k)} - \alpha /\zeta)_+ + 2\gamma \sum_{i=1}^{k-1} (d_{(i)} - 2\alpha/ \zeta)_+} \, \vee \, n^{1 +\gamma \sum_{i=1}^k (d_i - 2\alpha/ \zeta)_+ } \Bigr). 
\end{align*}
If $2\alpha /\zeta > d_{(k)}$, then we have for $\gamma \in (0,1]$,
\begin{align*}
 \E (S_n) \sim \E(\Tngamma) \sim  \biggl( \frac{2^{d-1}}{(d-1)\kappa_{d-2}}\biggr)^{k-1} \alpha^k \prod_{j=1}^k\,  \Bigl(\,  \alpha - \frac{\zeta}{2} d_j \, \Bigr)^{-1} \nu^{k-1}n, \ \ \ n\to\infty.
\end{align*}
If $\alpha / \zeta > d_{(k)}$, then we have for $\gamma \in (0,1]$,
$$
 \V(S_n) \sim \V (\Tngamma) \sim C^* n,  \ \ \ n\to\infty.
$$
\end{corollary}
To avoid repetition, we have not stated the CLT again but the CLT in Theorem \ref{t:clt.Tngamma} holds under the assumptions of Corollary \ref{cor:specialcase} with $c = \nu$. As mentioned in the introduction, \eqref{eqn:claim5.2} is a partial generalization of \citep[Claim 5.2]{candellero:fountoulakis:2016} to higher dimensions. 
\section{{\bf Proofs}}
\label{sec:proof} 

We first prove the basic lemma on approximating the hyperbolic probability density. Subsequently,  Section \ref{sec:hyperbolicapp} presents a few lemmas concerned with hyperbolic distance. Utilizing these lemmas, we prove  expectation and variance results for the sub-tree counts in Section \ref{sec:proofs_exp}. Finally Section \ref{sec:proofs_clt} establishes the required central limit theorem. Throughout this section, $C^*$ denotes a generic positive constant, which may vary between lines and does not depend on $n$.
\begin{proof}[Proof of Lemma \ref{l:pdf}]
We see that 
\begin{equation} \label{e:original.pdf}
\barrhona (t) = \frac{e^{\alpha (d-1)(R_n-t)} \bigl( 1-e^{-2\alpha (R_n-t)} \bigr)^{d-1} }{\int_0^{R_n} e^{\alpha (d-1)s} \bigl( 1-e^{-2\alpha s} \bigr)^{d-1} ds}. 
\end{equation}
Applying the binomial expansion
$$
 \bigl( 1-e^{-2\alpha s} \bigr)^{d-1} = \sum_{k=0}^{d-1} \begin{pmatrix} d-1 \\ k \end{pmatrix} (-1)^k e^{-2\alpha ks}, \, \, s > 0, 
$$
we find that $\int_0^{R_n} e^{\alpha (d-1)s}ds$ is asymptotically the leading term in the denominator in \eqref{e:original.pdf}. Due to an obvious inequality $ \bigl( 1-e^{-2\alpha (R_n-t)} \bigr)^{d-1} \leq 1$, we complete the proof of $(i)$. 

For the proof of $(ii)$, we need to handle the numerator of \eqref{e:original.pdf} as well. By another application of the binomial expansion, we have that
\begin{align*}
e^{\alpha (d-1) (R_n-t)}\, \Bigl[ \, 1+ \sum_{k=1}^{d-1} \begin{pmatrix} d-1 \\ k \end{pmatrix} (-1)^k e^{-2\alpha k(R_n-t)} \, \Bigr] = e^{\alpha (d-1) (R_n-t)} \bigl(1+o(1)  \bigr)
\end{align*}
uniformly for $0 \leq t \leq \lambda R_n$, where $0 < \lambda <1$. 
\end{proof}

\subsection{Lemmas on hyperbolic distances}
\label{sec:hyperbolicapp}~\\

We first mention some lemmas that help us to approximate the Poincar\'e metric. Given $u_1,u_2 \in B(0,R_n)$, let $\ta_{12} \in [0,\pi]$ be the relative angle between two vectors $\overrightarrow {Ou_1}$ and $\overrightarrow {Ou_2}$, where $O$ denotes the origin of $\Bdzeta$. We also denote $t_i = R_n-d(0,u_i)$, $i=1,2$.
\begin{lemma} \label{l:angle}
Set $\hat \ta_{12} = \bigl(e^{-2 \zeta(R_n-t_1)}+  e^{-2 \zeta(R_n-t_2)} \bigr)^{1/2}$. If $\hat \ta_{12} / \ta_{12}$ vanishes as $n\to\infty$, 
$$
d(u_1,u_2) = 2R_n -(t_1+t_2) + \frac{2}{\zeta}\log \sin \Bigl( \frac{\ta_{12}}{2} \Bigr) + O \biggl( \Bigl( \frac{\hat \ta_{12}}{\ta_{12}} \Bigr)^2 \biggr),  \ \ \ n\to\infty
$$
uniformly for all $u_1,u_2$ with $t_1+ t_2 \leq R_n-\omega_n$, where $\omega_n=\log \log R_n$.  
\end{lemma}
\begin{proof}
Fix a great circle of $B(0,R_n)$ spanned by $\overrightarrow{Ou_1}$ and $\overrightarrow{Ou_2}$. Then, the hyperbolic law of cosine yields
\begin{align*}
\cosh  \bigl( \zeta d(u_1,u_2) \bigr) &= \cosh  \zeta(R_n-t_1) \cosh  \zeta(R_n-t_2)  - \sinh \zeta(R_n-t_1)\sinh  \zeta(R_n-t_2)  \cos(\ta_{12}). 
\end{align*}
Since this great circle is a two-dimensional subspace of $B(0,R_n)$, the rest of the argument is completely the same as Lemma 2.3 in \citep{fountoulakis:2012}. 
\end{proof}
\begin{lemma} \label{l:int.angle}
In terms of the hyperbolic polar coordinate, let $X_1 = (t_1,\Ta_1), X_2 = (t_2,\Ta_2)$, where $\Ta_1,\Ta_2$ are iid random vectors on $C_d$ with density $\pi$, and $t_1,t_2$ are deterministic, representing the hyperbolic distance from the boundary. Under the setup in Lemma \ref{l:angle}, 
$$ 
\P\bigl(\, d(X_1,X_2) \leq R_n\bigr) \sim \frac{2^{d-1}}{(d-1)\kappa_{d-2}}e^{-\zeta(d-1)(R_n - t_1 - t_2)/2}, \ \ n\to\infty,
$$
uniformly on $\bigl\{(t_1,t_2): t_1+t_2 \leq R_n-\omega_n \bigr\}$, where $\kappa_{d-2} = \int_0^\pi \sin^{d-2} \ta\, d\ta$. 
\end{lemma}
\begin{proof}
First, let us denote $\Ta_{12}$ to be the relative angle between $\Ta_1$ and $\Ta_2$. Then, because of the uniformity of the angular density of $\Ta_i$ in \eqref{e:angular.pdf}, we can derive that the density $\pi_{rel}$ of $\Ta_{12}$, which is also the same as the conditional density of $\Ta_{12} | \Ta_1$, is given by 
\begin{equation}
\label{e:relangle}
\pi_{rel}(\ta) := (\kappa_{d-2})^{-1}\sin^{d-2} \theta, \, \ \theta \in [0,\pi].
\end{equation}
From the hyperbolic law of cosines, we know that the distance between $X_1$ and $X_2$ is determined by $t_1,t_2$ and their relative angle $\Ta_{12}$. Since $t_1,t_2$ are fixed, we can write
\begin{align*}
\P\bigl( \, d(X_1,X_2)\leq R_n \bigr) &= \int_{C_d^2} \one \bigl\{  \, d(u_1,u_2) \leq R_n \bigr\}\, \pi(\ta_1)\pi(\ta_2)\, d\ta_1 d\ta_2 \\
&= \frac{1}{\kappa_{d-2}}\, \int_0^\pi \one \bigl\{  \, d(u_1,u_2) \leq R_n \bigr\}\, \sin^{d-2} \ta_{12} d\ta_{12},
\end{align*}
for which we denote $u_i=(t_i,\ta_i)$, $i=1,2$, and $\ta_{12}$ is the relative angle between $u_1$ and $u_2$.
%
%
Now, we shall approximate the above integral. Let $A_{12} = e^{\zeta (R_n-t_1-t_2)/2}$. Claim 2.5 in \citep{fountoulakis:2012} proves that $A_{12}^{-1} / (\omega_n \hat \ta_{12}) \to \infty$ as $n\to\infty$. By virtue of Lemma \ref{l:angle}, along with $A_{12}^{-1} \to 0$ as $n\to\infty$, we have, on the set $\{ (t_1,t_2): t_1+t_2 \leq R_n-\omega_n \}$,
\begin{align*}
\int_{A_{12}^{-1}/\omega_n}^\pi \one \bigl\{\,  d(u_1,u_2) \leq R_n \, \bigr\}\, \sin^{d-2} \ta_{12}\, d\ta_{12} &\sim \int_{A_{12}^{-1}/\omega_n}^\pi \one \bigl\{\,  \sin \bigl(\frac{\ta_{12}}{2}\bigr) \leq A_{12}^{-1} \, \bigr\}\,\sin^{d-2} \ta_{12}\, d\ta_{12} \\
\sim \int_{A_{12}^{-1}/\omega_n}^\pi \one \{ \ta_{12} \leq 2 A_{12}^{-1} \}\, (\ta_{12})^{d-2}\, d\ta_{12} \, & \sim \, \frac{(2A_{12}^{-1})^{d-1}}{d-1}. 
\end{align*}
Therefore, 
\begin{align*}
\int_{0}^\pi \one \bigl\{\,  d(u_1,u_2) \leq R_n \, \bigr\}\, \sin^{d-2}\ta_{12}\, d\ta_{12} 
 &= o(A_{12}^{-(d-1)}) + \int_{A_{12}^{-1}/\omega_n}^\pi \one \bigl\{\,  d(u_1,u_2) \leq R_n \, \bigr\}\, \sin^{d-2}\ta_{12}\, d\ta_{12} \\
&\sim \frac{(2A_{12}^{-1})^{d-1}}{d-1}, \ \ \ n\to\infty
\end{align*}
as required. 
\end{proof}

We now present a crucial lemma that explains why the tree assumption is crucial to our asymptotics.

\begin{lemma} \label{l:int_angle_tree}
Let $\Gamma_k$ be a tree on $k$ vertices with edge set $E$. Let $X_1,\dots, X_k$ be iid random variables with common density $\rho_{n,\alpha} \times \pi$. Define $T_i=R_n-d(0,X_i)$, $i=1,\dots,k$, and write $\BT = (T_1,\dots,T_k)$. Then, it holds that 
$$
\P \bigl( \, d(X_i,X_j)\leq R_n, \ (i,j)\in E\, | \, \BT \bigr) = \prod_{(i,j)\in E} \P \bigl( \, d(X_i,X_j)\leq R_n \, |\, \BT \bigr) \ \ a.s.
$$
\end{lemma}

\begin{proof}
Let $\Ta_i$ be the angular part of $X_i$, $i=1,\dots,k$. For the proof, it suffices to show that 
\begin{equation}  \label{e:induction}
\P \bigl( \, d(X_i,X_j)\leq R_n, \ (i,j)\in E\, | \, \Ta_1, \BT \bigr) = \prod_{(i,j)\in E} \P \bigl( \, d(X_i,X_j)\leq R_n \, |\, \BT \bigr) \ \ a.s.
\end{equation}
We again denote the relative angle between $\Ta_1, \Ta_2$ as $\Ta_{12}$. From \eqref{e:relangle}, we have that $\Ta_{12} \stackrel{d}{=} \Ta_{12} | \Ta_1$, i.e., the conditional distribution of  $\Ta_{12}$ conditioned on $\Ta_1$ is same as its unconditional distribution. From the hyperbolic law of cosines (see Lemma \ref{l:angle}), we know that $d(X_1,X_2)$ depends only on $T_1,T_2,\Theta_{12}$, and so, from the above equality of distributions, we obtain that
\begin{equation}  \label{e:uniformity.angular}
\P\bigl( \, d(X_i,X_j) \leq R_n\, | \, \Ta_i, \BT \bigr) = \P \bigl( \, d(X_i,X_j) \leq R_n\, | \, \BT \bigr) \ \ a.s.
\end{equation}
Suppose that $\Gamma_k$ is of depth $1$ rooted at vertex $1$. In this case, since $X_2,\dots,X_k$ are iid, 
\begin{align*}
\P \bigl( \, d(X_1,X_j) \leq R_n, \ j=2,\dots,k\, |\, \Ta_1, \BT \bigr) &= \prod_{j=2}^k \P\bigl( \, d(X_1,X_j)\leq R_n \, | \Ta_1, \BT) \\
&= \prod_{j=2}^k \P\bigl( \, d(X_1,X_j)\leq R_n \, |  \BT) \ \ a.s.
\end{align*}
Suppose, for induction, \eqref{e:induction} holds for any tree $\Gamma_k$ rooted at vertex $1$ with depth $M \geq 1$. Assume subsequently that $\Gamma_k$ is rooted at vertex $1$ with depth $M+1$. Let $2,\dots,m$ be the vertices connected to $1$, and $S_2,\dots,S_m$ the corresponding trees rooted at $2,\dots,m$. Let $E(S_\ell)$ be the edge set of $S_\ell$. Then, from the disjointness of the trees and independence of $X_i$'s, we have
\begin{align*}
\P \bigl( \, d(X_i,X_j) \leq R_n, \ (i,j)\in E\, |\, \Ta_1, \BT \bigr) 
& = \prod_{\ell=2}^m \P\bigl( \,  d(X_i,X_j) \leq R_n, \ (i,j)\in E(S_\ell)\, | \, \Ta_1, \BT \bigr) \ \ a.s.
\end{align*}
Now, an application of conditional expectations for each $\ell=2,\dots,m$, gives us the desired result as follows:
\begin{align*}
&\P\bigl( \,  d(X_i,X_j) \leq R_n, \ (i,j)\in E(S_\ell)\, | \, \Ta_1, \BT \bigr) \\
&= \E \Bigl[ \, \one\bigl\{ d(X_1,X_\ell) \leq R_n \bigr\}\, \P \bigl( \, d(X_i,X_j)\leq R_n, \ (i,j)\in E(S_\ell)\setminus \{(1,\ell) \}\, | \, \Ta_\ell, \BT \bigr) \bigl| \, \Ta_1, \BT \Bigr] \\
&= \E \Bigl[  \, \one\bigl\{ d(X_1,X_\ell) \leq R_n \bigr\} \prod_{(i,j)\in E(S_\ell)\setminus \{ (1,\ell) \}} \P \bigl(  \, d(X_i,X_j)\leq R_n\, |\, \BT \bigr) \bigl| \, \Ta_1,\BT  \Bigr] \\
&=\prod_{(i,j)\in E(S_\ell)} \P \bigl( \, d(X_i,X_j)\leq R_n\, | \, \BT \bigr) \ \ a.s., 
\end{align*}
where the induction hypothesis is applied for the second equality, and \eqref{e:uniformity.angular} is used for the third equality. 
\end{proof}
\subsection{Proofs of the expectation and variance results}
\label{sec:proofs_exp}
In what follows, since we plan to calculate moments for $\Tngamma$, we shall introduce some shorthand notations to save spaces. For $u_i \in B(0,R_n)$, $i=1,\dots,k$, we write
\begin{equation}
\label{eqn:gngamma}
\gngamma (u_1,\dots,u_k) := \one \bigl\{\, 0<d(u_i,u_j) \leq R_n,  \ (i,j)\in E, \ t_i \leq \gamma R_n, \ i=1,\dots,k\, \bigr\}, 
\end{equation}
and
\begin{equation*}
\hngamma (u_1,\dots,u_k) := \one \bigl\{\, 0<d(u_i,u_j) \leq R_n,  \ (i,j)\in E, \ t_i > \gamma R_n \ \text{for some }  i=1,\dots,k\, \bigr\},
\end{equation*}
where $t_i=R_n-d(0,u_i)$. 
\begin{proof}[Proof of Theorem \ref{t:expectation.Tngamma}]
Let $X_1,\ldots,X_k$ be iid random variables with density $\rho_{n,\alpha} \times \pi$ and set $T_i = R_n - d(0,X_i)$ as before. By an application of the Palm theory (Lemma \ref{l:palm1}),
\begin{align}
\E ( \Tngamma ) &= n^k \, \E\bigl( g_{n,\gamma}(X_1,\ldots,X_k)\bigr) \label{e:splitting}  \\
&=  n^k\,  \E \bigl[ g_{n,\gamma}(X_1,\ldots,X_k)  \prod_{(i,j) \in E} \one \{ T_i + T_j \leq R_n - \omega_n \} \bigr] \notag \\
&\qquad + n^k\, \E \bigl[ g_{n,\gamma}(X_1,\ldots,X_k)  \one\bigl\{ \cup_{(i,j) \in E} \{T_i+T_j > R_n-\omega_n\} \bigr\} \bigr]  \notag \\
&:= A_n + B_n, \notag
\end{align}
where $\omega_n=\log \log R_n$. From now, the argument aims to show that $A_n$ coincides asymptotically with the right-hand side of \eqref{e:expectation.Tngamma}.
It follows from the conditioning on the radial distances of $X_1,\dots,X_k$ from the boundary and Lemma \ref{l:int_angle_tree}  that 
\begin{equation*}
A_n = n^k \E \Bigl[\ \one \bigl\{ T_i +T_j \leq R_n-\omega_n, \ (i,j)\in E, \ T_i \leq \gamma R_n, \ i=1,\dots,k \bigr\} \prod_{(i,j)\in E} \P\bigl( \, d(X_i,X_j)\leq R_n\, |\, \BT\bigr) \Bigr],
\end{equation*}
where $\BT=(T_1,\dots,T_k)$. Setting $\bt = (t_1,\dots,t_k)$, we may write
\begin{equation}
A_n = n^k\, \ingam dt_1 \cdots \ingam dt_k\, \one \bigl\{ t_i + t_j \leq R_n-\omega_n, \ (i,j)\in E\, \bigr\} \prod_{(i,j)\in E} \P\bigl( d(X_i,X_j) \leq R_n\, |\, \bt \bigr) \, \barrhona (\bt), \label{e:int.An}
\end{equation}
where $\barrhona (\bt)$ denotes the product of densities in \eqref{e:barrho}:
$$
\barrhona (\bt) := \prod_{i=1}^k \barrhona (t_i). \ \ 
$$
Applying Lemma \ref{l:int.angle} on the set $\bigl\{ t_i+t_j \leq R_n -\omega_n, \ (i,j)\in E \bigr\}$, we have, as $n\to\infty$,
\begin{align*}
\prod_{(i,j)\in E} \P(d(X_i,X_j)\leq R_n\, |\, \bt) &\sim \prod_{(i,j)\in E} \frac{2^{d-1}}{(d-1)\kappa_{d-2}}\, e^{-\zeta (d-1) (R_n-t_i-t_j)/2} \\
&= \Bigl( \frac{2^{d-1}}{(d-1)\kappa_{d-2}} \Bigr)^{k-1} e^{-[(k-1)R_n  -  \sum_{i=1}^k d_it_i]\zeta (d-1)/2}
\end{align*}
Further, it follows from Lemma \ref{l:pdf} $(ii)$ that 
$$
\barrhona (\bt) \sim \alpha^k (d-1)^k e^{-\alpha (d-1) \sum_{i=1}^k t_i}, \ \ n\to\infty
$$
uniformly for $t_i \leq \gamma R_n$, $i=1,\dots,k$. Substituting these results into \eqref{e:int.An}, we obtain
\begin{equation}
\label{e.An}
A_n \sim \Bigl( \frac{2^{d-1}}{\kappa_{d-2}} \Bigr)^{k-1} \alpha^k (d-1) \, n^k e^{-\zeta(d-1)(k-1) R_n/2} \prod_{i=1}^k \angamma(d_i). 
\end{equation}

We now show that $B_n = o(A_n)$ as $n\to\infty$, unless $\gamma =1/2$. We only check the case in which there is an edge joining $X_1$ and $X_2$ such that $T_1+T_2 > R_n-\omega_n$, while all the other edges satisfy $T_i+T_j \leq R_n-\omega_n$. However, the following argument can apply in an obvious way, even when multiple edges satisfy $T_i + T_j >R_n-\omega_n$. Specifically, we shall verify 
\begin{align}
C_n &:= n^k\P \bigl( d(X_i,X_j) \leq R_n, \ (i,j)\in E, \ T_i \leq \gamma R_n, \ i=1,\dots,k, \label{e:term.Cn} \\
&\qquad \qquad \qquad T_1+T_2 > R_n-\omega_n, \ T_i + T_j \leq R_n-\omega_n, \ (i,j)\in E\setminus \{  (1,2)\}  \bigr) = o(A_n)\notag
\end{align}
Proceeding in the same way as the computation for $A_n$, while applying an obvious bound \\
$\one \bigl\{ d(X_1,X_2) \leq R_n \bigr\} \leq 1$, we see that 
\begin{align*}
C_n &\leq C^* n^k e^{-\zeta(d-1)(k-2)R_n/2} \prod_{i=3}^n \angamma (d_i) \\
&\quad \times \ingam \ingam e^{2^{-1}\zeta (d-1)\sum_{i=1}^2 (d_i - 1-2\alpha/ \zeta)t_i } \one \bigl\{ t_1+t_2 >R_n-\omega_n \bigr\}dt_1 dt_2.  
\end{align*}
By comparing this upper bound with the right-hand side of \eqref{e:expectation.Tngamma}, it suffices to show that 
\begin{align}
D_n &:= e^{\zeta(d-1) R_n/2}\ingam \ingam e^{2^{-1} \zeta(d-1) \sum_{i=1}^2 (d_i - 1-2\alpha /\zeta)t_i } \one \bigl\{ t_1+t_2 >R_n-\omega_n \bigr\}dt_1 dt_2 \label{e:close.boundary} \\
&= o \bigl( \angamma (d_1)\angamma (d_2) \bigr), \ \ n\to\infty. \notag
\end{align}
If $0<\gamma<1/2$, then $D_n$ is identically $0$; hence, we may restrict ourselves to the case $1/2 < \gamma <1$.
Without loss of generality, we may assume that $d_1 \geq d_2$, and rewrite $D_n$ as 
$$
D_n = e^{\zeta (d-1)R_n/2} \int_{(1-\gamma)R_n-\omega_n}^{\gamma R_n} e^{\zeta (d-1)(d_1 - 1 -2\alpha /\zeta)t_1/2} \int_{R_n-t_1-\omega_n}^{\gamma R_n} e^{\zeta (d-1)(d_2 - 1 -2\alpha /\zeta)t_2/2}dt_2 dt_1. 
$$
If $0 < 2 \alpha/\zeta< d_2-1$,
\begin{align*}
D_n & \leq C^* e^{ \zeta (d-1) [R_n + \sum_{i=1}^2 (d_i-1 -2\alpha /\zeta)\gamma R_n]/2} \\ 
&  = e^{ (1/2-\gamma)\zeta (d-1)R_n} O\bigl( \angamma (d_1)\angamma (d_2) \bigr) = o \bigl( \angamma (d_1)\angamma (d_2) \bigr). 
\end{align*}
On the contrary, let $2 \alpha /\zeta> d_2-1$. Then, 
\begin{align*}
D_n & \leq C^* e^{\zeta(d-1)[ (d_2-2\alpha /\zeta)R_n + (d_1-d_2)\gamma R_n]/2 } = o \bigl( \angamma (d_1)\angamma (d_2) \bigr).
\end{align*}
The same result can be obtained as well in the boundary case $2\alpha /\zeta= d_2-1$, and thus, we have proven \eqref{e:close.boundary}. Finally, if $\gamma=1/2$, the above calculations imply
$$
D_n = O \Bigl( \angamma(d_1) \angamma(d_2) \Bigr), \ \ n\to\infty,
$$
and thus, \eqref{e:exp.gamma.one.half} follows. 
\vspace{5pt}

Subsequently, we proceed to proving \eqref{e:expectation.Sn}. For $0<\gamma<1$, define $\Ungamma = S_n - \Tngamma$. Suppose $2 \alpha/\zeta > d_{(k)}$. Appealing to \eqref{e:expectation.Tngamma}, we obtain
$$
\E (  \Tngamma ) \sim \biggl( \frac{2^{d-1}}{(d-1)\kappa_{d-2}}\biggr)^{k-1} \alpha^k  \, \prod_{i=1}^k \bigl( \alpha -  \zeta d_i/2 \bigr)^{-1} n^k e^{- \zeta(d-1) (k-1)R_n/2},
$$
because, for every $i=1,\dots,k$, $\angamma (d_i)$ converges to $(d-1)^{-1}\bigl( \alpha  - \zeta d_i/2 \bigr)^{-1}$. Therefore, to complete our proof, it suffices to verify that  
\begin{equation}  \label{e:EUngamma}
\E ( \Ungamma )= o\bigl( n^k e^{-\zeta(d-1) (k-1)R_n/2} \bigr). 
\end{equation}
Another application of the Palm theory (Lemma \ref{l:palm1}) with $\omega_n=\log \log R_n$ yields,
\begin{align*}
\E ( \Ungamma ) &= n^k\, \E \bigl( \hngamma (X_1,\dots,X_k) \bigr) \\
&= n^k\, \E \bigl[ \hngamma (X_1,\dots,X_k) \prod_{(i,j)\in E}\one \{ T_i + T_j \leq R_n - \omega_n \} \bigr]  \\
&\qquad +  n^k\, \E \bigl[ \hngamma(X_1,\ldots,X_k)  \one \bigl\{ \cup_{(i,j) \in E} \{T_i+T_j > R_n-\omega_n\} \bigr\} \bigr]  \\
&:= A_n^{\prime} + B_n^{\prime}.
\end{align*}

We can calculate $A_n^\prime$ in almost the same manner as $A_n$; the only difference is that when handling $\barrhona (\bt)$, the present argument needs to apply the inequality in Lemma \ref{l:pdf} $(i)$, that is,
$$
\barrhona (\bt) \leq C^* \alpha^k (d-1)^k e^{-\alpha (d-1) \sum_{i=1}^kt_i} \ \text{uniformly for } 0<t_i<R_n, \ i=1,\dots,k.
$$
Consequently, we obtain
\begin{align*}
A_n^\prime &\leq C^* n^k e^{- \zeta(d-1)(k-1)R_n/2} \int_0^{R_n} dt_1 \cdots \int_0^{R_n} dt_k \\
& \qquad \qquad \times \one \{ \, t_i > \gamma R_n \ \text{for some } i=1,\dots,k\, \}\, e^{2^{-1} \zeta (d-1)\sum_{i=1}^k (d_i - 2\alpha/ \zeta)t_i} \\
&= o\bigl( n^k e^{- \zeta(d-1)(k-1)R_n/2} \bigr).
\end{align*}
Here, the second equality follows from the assumption that $d_i - 2\alpha /\zeta < 0$ for all $i=1,\dots,k$. Furthermore, we can show that $B_n^{\prime} = o \bigl( n^k e^{-\zeta(d-1)(k-1)R_n/2} \bigr)$, the proof of which is similar to the corresponding result for the derivation of \eqref{e:term.Cn} and \eqref{e:close.boundary}, so we omit it. Thus, we have \eqref{e:EUngamma} as needed to complete the proof of the theorem.
\end{proof}
\begin{proof}[Proof of Corollary \ref{cor:expectation.simple}]
The sub-tree counts $\Tngamma$ relating to a sub-tree with $\Dk = k-1$, $d_{(k-1)} = \cdots = \Done = 1$ is given by 
$$
\Tngamma = \sumXkPn \hspace{-10pt}\one \bigl\{ 0<d(X_1, X_i) \leq R_n, \ i=2,\dots,k, \ \ T_i \leq \gamma R_n, \ i=1,\dots,k \bigr\}.
$$ 
We also define $\Ungamma = S_n - \Tngamma$. First, \eqref{e:special.case.ESn} is a direct consequence of \eqref{e:expectation.Sn}, so we shall prove only (ii) and (iii). 

As for the proof of $(ii)$, we start with deriving a suitable upper bound for $\E ( \Ungamma)$. By the Palm theory (Lemma \ref{l:palm1}) and Lemma \ref{l:pdf} $(i)$, 
\begin{align*}
\E(\Ungamma ) &= n^k \, \P \bigl( d(X_1,X_i) \leq R_n, \ i=2,\dots,k, \ T_i > \gamma R_n \ \text{for some } i=1,\dots,k\, \bigr)\\
&\leq n^k k \, \P ( T_1 > \gamma R_n\, ) \leq C^* n^k e^{-\alpha (d-1) \gamma R_n}R_n. 
\end{align*}
Taking logarithm on both sides, we get
$$ \limsup_{n\to\infty} \frac{\log \E ( \Ungamma)}{R_n \vee \log n} \leq k(c\vee 1)^{-1} -\alpha (d-1) \gamma (1\vee c^{-1})^{-1}. $$
On the other hand, it follows from Theorem \ref{t:expectation.Tngamma} that 
\begin{align*}
\E (  \Tngamma ) &\sim C^* n^k e^{- \zeta(d-1)(k-1) R_n/2} \angamma (k-1) \sim C^* n^k e^{-\alpha (d-1) \gamma R_n - \zeta (d-1)(k-1) (1-\gamma) R_n/2},
\end{align*}
which implies that, as $n\to\infty$, 
$$
\frac{\log \E ( \Tngamma)}{R_n \vee \log n} \to k(c\vee 1)^{-1} - \alpha (d-1) \gamma (1\vee c^{-1})^{-1} - \zeta (d-1)(k-1)(1-\gamma)(1\vee c^{-1})^{-1}/2. 
$$
Moreover, if $\limsup_{n\to\infty} \E(\Ungamma) / \E(\Tngamma) < \infty$, 
$$
\limsup_{n\to\infty} \frac{\log \E (S_n)}{R_n \vee \log n} = \limsup_{n\to\infty} \frac{\log E(\Tngamma)}{R_n \vee \log n},
$$
and if $\limsup_{n\to\infty} \E(\Ungamma) / \E(\Tngamma) = \infty$,
$$
\limsup_{n\to\infty} \frac{\log \E (S_n)}{R_n \vee \log n} = \limsup_{n\to\infty} \frac{\log E(\Ungamma)}{R_n \vee \log n}.
$$
Therefore, using obvious inequalities
\begin{align*}
\liminf_{n\to \infty} \frac{\log \E ( \Tngamma )}{R_n\vee \log n} &\leq \liminf_{n\to\infty} \frac{\log \E (S_n)}{R_n \vee \log n} \leq \limsup_{n\to\infty} \frac{\log \E (S_n)}{R_n\vee \log n}\\
& \leq \limsup_{n\to\infty} \frac{\log \E ( \Tngamma )}{R_n\vee \log n} \vee \limsup_{n\to\infty} \frac{\log \E (\Ungamma )}{R_n\vee \log n},
\end{align*}
and letting $\gamma \nearrow 1$, we obtain 
$$ \frac{\log \E (S_n)}{R_n \vee \log n}  \to k(c\vee 1)^{-1} -\alpha (d-1) (1\vee c^{-1})^{-1}, \ \ \text{as } n\to\infty. $$
The proof of $(iii)$ is very similar to that of $(ii)$, so we omit it. 
\end{proof}

\begin{proof}[Proof of Theorem \ref{t:variance.Tngamma}]
We start by writing
\begin{align*}
\E \bigl[ (\Tngamma)^2 \bigr]&= \sum_{\ell =0}^k \E \Bigl[ \sum_{\X \in \mathcal P_{n,\neq}^k} \sum_{\X^\prime \in \mathcal P_{n,\neq}^k} \gngamma (\X)\, \gngamma(\X^\prime) \one \bigl\{ |\X \cap \X^\prime| = \ell \bigr\} \Bigr] := \sum_{\ell = 0}^k \E ( I_\ell ). 
\end{align*}
For $\ell = 0$, applying the Palm theory (Lemma \ref{l:palm1}),  
\begin{align*}
\E (I_0 ) &= n^{2k}\, \Bigl[ \E \bigl( \gngamma (\X)\bigr) \Bigr]^2 = \bigl[ \E ( \Tngamma ) \bigr]^2.
\end{align*}

Let $\Gamma^{(i,j)}_{2k-1}$ be a tree on $[2k-1]$ (meaning $|E| = 2k-2$) formed by taking two copies of $\Gamma_k$ and identifying the vertex of degree $d_i$ in one copy with the vertex of degree $d_j$ in the other copy. In other words, the degree sequence of $\Gamma^{(i,j)}_{2k-1}$ is $d_i + d_j, d_i, d_j$ and a pair of $d_\ell$'s for $\ell \in [k] \setminus \{i,j\}$. We first note that $I_1 =  \sum_{i,j=1}^k \C\bigl(\Gamma^{(i,j)}_{2k-1},HG^{(\gamma)}_n(R_n;\alpha,\zeta)\bigr)$, and so, from the identity for $E(I_0)$ above, we have that 
\begin{align}
\V(\Tngamma) &= \sum_{\ell =1}^k \E( I_\ell ) \geq \E ( I_1 ) \geq E\Bigl[\, \C\bigl(\Gamma^{(i^\prime, j^\prime)}_{2k-1},HG^{(\gamma)}_n(R_n;\alpha,\zeta)\bigr)\Bigr], \label{e:2k-1.term} 
\end{align}
where $i^\prime = j^\prime = (k)$, that is, $d_{i^\prime} = d_{j^\prime} = d_{(k)}$.  Therefore, applying Theorem \ref{t:expectation.Tngamma} to \eqref{e:2k-1.term}, we derive that

\begin{equation}  \label{e:lower1.variance}
\V(\Tngamma)  = \Omega \biggl(n^{2k-1} e^{-\zeta(d-1) (k-1)R_n} \angamma (2 \Dk) \prod_{i=1}^{k-1} \angamma (\Di)^2\biggr). 
\end{equation}
Note that due to \eqref{e:splitting} and \eqref{e.An}, the assumption $\gamma \neq 1/2$ is not required for the lower bound.  Similarly, by using the bound $I_k \geq \C\bigl(\Gamma_k,HG^{(\gamma)}_n(R_n;\alpha,\zeta)\bigr)$, we get
\begin{align*}
\V(\Tngamma) \geq \E ( I_k ) \geq  \E\Bigl[\, \C\bigl(\Gamma_k,HG^{(\gamma)}_n(R_n;\alpha,\zeta)\bigr)\Bigr] \label{e:k.term}.
\end{align*}
Thus, again from Theorem \ref{t:expectation.Tngamma}, we derive that
\begin{equation}  \label{e:lower2.variance}
\V(\Tngamma)= \Omega\biggl( n^k e^{-\zeta(d-1) (k-1)R_n/2} \prod_{i=1}^k \angamma (d_i)\biggr). 
\end{equation}
Finally, combining  \eqref{e:lower1.variance} and \eqref{e:lower2.variance} establishes \eqref{e:variance.Tngamma}.  

We now proceed to show \eqref{e:variance.Sn}. Assume that $\alpha  /\zeta > \Dk$, and for $0<\gamma <1$, write $\V(\Tngamma) = \sum_{\ell = 1}^k \E( I_\ell )$ as in \eqref{e:2k-1.term}. To derive exact asymptotics for $\V(\Tngamma)$, we shall derive exact asymptotics for $\E( I_1)$ and $\E( I_k)$ and also show that $\E(I_{\ell}) = o\bigl(\E( I_1) \vee \E( I_k)\bigr)$ for $2 \leq \ell \leq k-1$.

From the representation of $I_1$ before \eqref{e:2k-1.term} and Theorem \ref{t:expectation.Tngamma}, we have that 

\begin{align*}
\E ( I_1 ) &= \sum_{i,j=1}^k \, E \Bigl[ \, \C\bigl(\Gamma^{(i,j)}_{2k-1},HG^{(\gamma)}_n(R_n;\alpha,\zeta)\bigr) \Bigr]  \\
&\sim  C^* n^{2k-1} e^{-\zeta(d-1) (k-1)R_n} \sum_{i,j=1}^k \prod_{\ell = 1, \, \ell \neq i,j}^k \angamma (d_\ell)^2 \angamma (d_i+d_j) \angamma(d_i) \angamma (d_j). 
\end{align*}
However, due to the constraint $\alpha  /\zeta > \Dk $, we see that 
$$
\prod_{\ell = 1, \, \ell \neq i,j}^k \angamma (d_\ell)^2 \angamma (d_i+d_j) \angamma(d_i) \angamma (d_j)
$$
tends to a positive constant for all $i,j = 1,\dots,k$. Thus, we conclude that 
$$
\E ( I_1) \sim C^* n^{2k-1} e^{-\zeta(d-1) (k-1)R_n}, \ \ \text{as } n\to\infty.
$$
Similarly, we can verify that $\E(I_k) \sim C^* n^{k} e^{-\zeta (d-1) (k-1)R_n/2} \ \ \text{as } n\to\infty$. 

Subsequently, we investigate the rate of $\E(I_\ell)$ for $\ell =2,\ldots,k-1$. Similar to $I_1$, we can express $I_\ell$ as a sum of $\C\bigl(H,HG^{(\gamma)}_n(R_n;\alpha,\zeta)\bigr)$ over subgraphs $H$ on $[2k-\ell]$ formed by identifying $\ell$ vertices on two copies of $\Gamma_k$. Choosing a spanning tree $\Gamma_H$ of $H$, along with the monotonicity of $\C$, we get $\C\bigl(H,HG^{(\gamma)}_n(R_n;\alpha,\zeta)\bigr) \leq \C\bigl(\Gamma_H,HG^{(\gamma)}_n(R_n;\alpha,\zeta)\bigr)$. By the choice of $H$ and $\Gamma_H$,  the vertex degrees are necessarily smaller than $\alpha  / \zeta$.  Thus, from Theorem \ref{t:expectation.Tngamma}, we get that for every $\Gamma_H$, 
$$
\E \Bigl[\, \C\bigl(\Gamma_H,HG^{(\gamma)}_n(R_n;\alpha,\zeta)\bigr)\Bigr] \sim C^* n^{2k-\ell} e^{-\zeta (d-1)(2k -\ell-1)R_n/2}. 
$$
Now, we derive that
\[ \E(I_\ell) = O\bigl(n^{2k-\ell} e^{-\zeta(d-1) (2k-\ell-1)R_n/2}\bigr), \, \, \ell = 2,\ldots,k-1.\]
Note that, for every $n\geq1$, $g(m) = n^m e^{-\zeta(d-1)(m-1)R_n/2}$ is monotonic in $m \in \bbn_+$, so either $\E( I_1 )$ or $\E ( I_k )$ determines the actual growth rate of $\V(\Tngamma)$. Thus, we have that $\E(I_\ell) = o\bigl(\E( I_1) \vee \E( I_k)\bigr)$ for all $2 \leq \ell \leq k-1$, and we can conclude that 
$$ \V(\Tngamma)\sim C^* \Bigl[ n^{2k-1}e^{-\zeta(d-1)(k-1)R_n} \vee  n^k e^{-\zeta (d-1)(k-1)R_n/2}  \Bigr], \ \ n\to\infty. $$

Next, let $\Ungamma = S_n - \Tngamma$ for $0 < \gamma <1$. We can finish the proof, provided that 
\begin{equation}  \label{e:variance.Ungamma}
\V(\Ungamma) = o \bigl(  n^{2k-1}e^{-\zeta(d-1)(k-1)R_n} \vee  n^k e^{-\zeta(d-1) (k-1)R_n/2}  \bigr). 
\end{equation}
As in the proof for $\V(\Tngamma)$, we write
$$ 
\V(\Ungamma) = \, \E \bigl[ \sum_{\X \in \mathcal P_{n,\neq}^k} \sum_{\X^\prime \in \mathcal P_{n,\neq}^k} \hngamma (\X)\, \hngamma(\X^\prime)\, \one \bigl\{ |\X \cap \X^\prime| = \ell \bigr\} \bigr] :=\sum_{\ell=1}^k \E( J_\ell ). 
$$
Repeating the same argument as the proof of \eqref{e:EUngamma} in Theorem \ref{t:expectation.Tngamma}, we obtain that for all $\ell=1,\dots,k$,
$$
\E ( J_\ell ) = o \bigl(  n^{2k-\ell} e^{-\zeta (d-1)(2k-\ell-1)R_n/2}\bigr).
$$
Thus, \eqref{e:variance.Ungamma} follows. 
Moreover, by the Cauchy-Schwarz inequality, 
\begin{align*}
\bigl| \, \text{Cov}(\Tngamma, \Ungamma) \, \bigr| &\leq \sqrt{\V(\Tngamma) \V(\Ungamma)} = o \bigl(  n^{2k-1}e^{-\zeta(d-1)(k-1)R_n} \vee  n^k e^{-\zeta (d-1)(k-1)R_n/2}  \bigr),
\end{align*} 
from which we have
\begin{align*}
\V(S_n) &= \V(\Tngamma) + \V(\Ungamma) +2\text{Cov}(\Tngamma,\Ungamma) \\
&\sim C^*\Bigl[  n^{2k-1}e^{-\zeta(d-1)(k-1)R_n} \vee  n^k e^{-\zeta(d-1) (k-1)R_n/2} \Bigr], \ \ \ n\to\infty.
\end{align*}
\end{proof}
\subsection{Proof of the central limit theorem} \label{sec:proofs_clt}~

The proof relies on the normal approximation bound given by Theorem \ref{thm:last2016} below, which itself is derived from Malliavin-Stein method in \citep{last:peccati:schulte:2016}.  We shall first introduce the normal approximation bound and then use it to prove our central limit theorem. 

\subsubsection{Malliavin-Stein bound for Poisson functionals} 
\label{sec:Malliavin}~

Malliavin-Stein method has emerged as a crucial tool in proving limiting approximation results for Gaussian, Poisson, and Rademacher functionals. We introduce some notation and the result of use to us. For a more detailed introduction to this subject, see \citep{last:penrose:2017, peccati:reitzner:2016}.  

Let $\cP$ be a Poisson point process on a finite measure space $\BX$ with intensity measure $\lambda(\cdot)$. For a functional $F$ of Radon counting measures (i.e., locally-finite collection of points), define the {\em Malliavin difference operators} as follows : For $x \in \BX$, the {\em first order difference operator} is $D_xF := F(\cP \cup \{x\}) - F(\cP)$. The {\em higher order difference operators} are defined inductively as $D^\ell_{x_1,\ldots,x_\ell} F:= D_{x_\ell}(D^{\ell-1}_{x_1,\ldots,x_{\ell-1}}F)$. We require only the first and second order difference operators. The latter is easily seen to be 
\[ D^2_{x,y}F = F(\cP \cup \{x,y\}) - F(\cP \cup \{y\}) - F(\cP \cup \{x\}) + F(\cP), \]
for $x,y \in \BX$. We say that $F \in dom \, D$ if 
\[ \E\bigl(F(\cP)^2\bigr) < \infty, \, \, \, \E \Bigl[ \, \int_{\BX} (D_xF(\cP))^2 \lambda(\md x) \Bigr] < \infty .\]
\begin{theorem}(\citep[Theorems 1.1 and 6.1]{last:peccati:schulte:2016})
	\label{thm:last2016}
	Let $F \in dom \, D$ and let $N$ be a standard normal random variable. Define
	\begin{equation*}
	c_1 := \sup_{x \in \BX} \E \bigl( |D_xF|^5\bigr) \, \, ; \, \, c_2 := \sup_{x,y \in \BX} \E \bigl( |D^2_{x,y}F|^5 \bigr)
	\end{equation*}
	where these supremums are essential supremums with respect to $\lambda$ and $\lambda^2$ respectively. Then,
	\begin{eqnarray*}
	d_W\left(\frac{F-\E(F)}{\sqrt{\V(F)}},N\right) & \leq &  W_1 + W_2 + W_3, \\
	d_K\left(\frac{F-\E(F)}{\sqrt{\V(F)}},N\right) & \leq & W_1 + W_2 + W_3 + W_4 + W_5 + W_6, 
	\end{eqnarray*}
	where $W_1,\dots,W_6$ are defined as follows :
	\begin{eqnarray*}
		W_1 & = & \frac{2(c_1c_2)^{1/5}}{\V(F)} \left[ \int_{\BX^3} [ \P(D^2_{x_1,x_3}F \neq 0)\P(D^2_{x_2,x_3}F \neq 0) ]^{1/20} \lambda^3\bigl(\md (x_1,x_2,x_3)\bigr) \right]^{1/2} \\
		W_2 & = & \frac{c_2^{2/5}}{\V(F)} \left[ \int_{\BX^3} [ \P(D^2_{x_1,x_3}F \neq 0)\P(D^2_{x_2,x_3}F \neq 0) ]^{1/10} \lambda^3\bigl(\md (x_1,x_2,x_3)\bigr) \right]^{1/2} \\
		W_3 & = & \frac{1}{\V(F)^{3/2}} \int_\BX \E\bigl(|D_xF|^3\bigr) \lambda(\md x). \\
		W_4 & = & \frac{c_1^{3/5}\lambda(\BX)}{\V(F)^{3/2}} + \frac{c_1^{4/5}\lambda(\BX)^{5/4} + 2 c_1^{4/5}\lambda(\BX)^{3/2}}{\V(F)^2} \\
		W_5 & = & \frac{c_1^{2/5}\lambda(\BX)^{1/2}}{\V(F)} \\
		W_6 & = & \frac{ \sqrt{6}(c_1c_2)^{1/5}  +  \sqrt{3}c_2^{2/5}}{\V(F)} \left[ \int_{\BX^2}  \P(D^2_{x_1,x_2}F \neq 0) ]^{1/10} \lambda^2\bigl(\md (x_1,x_2)\bigr) \right]^{1/2}. \\ 
	\end{eqnarray*}
\end{theorem}
The Wasserstein bound in \citep[Theorem 1.1]{last:peccati:schulte:2016} contains three terms. For the first two terms, we have used the bounds in the proof of \citep[Theorem 6.1]{last:peccati:schulte:2016} with $p_1 = p_2 = 1$, and we have left the third term unchanged. In the bound for the Kolmogorov distance, we have used the trivial inequality $\P(D_xF \neq 0) \leq 1$ in the bounds of \citep[Theorem 6.1]{last:peccati:schulte:2016}. We have also used the fact that for any $\ell \geq 1$ and distinct $x_1,\ldots,x_\ell \in \BX$, 
\[ D^\ell_{x_1,\ldots,x_\ell}\left( \frac{F-\E(F)}{\sqrt{\V(F)}}\right) = \frac{D^\ell_{x_1,\ldots,x_\ell}F}{\sqrt{\V(F)}}. \]
\subsubsection{Some auxilliary lemmas and the proof of CLT in  Theorem \ref{t:clt.Tngamma}.}
\label{sec:lemma_clt_proof}

Having already derived variance bounds in Theorem \ref{t:variance.Tngamma}, we shall successively compute the remaining bounds in Theorem \ref{thm:last2016}. Once again, $C^*$ is a positive generic constant, which is independent of $n$. For $0 < \gamma < 1/2$, the constants given in Theorem \ref{thm:last2016} are
$$
c_{1,n} = \sup_{x\in D_\gamma(R_n)} \E\bigl[|D_x\Tngamma|^5 \bigr], \ \ \ \ \ c_{2,n} = \sup_{x, y \in D_\gamma(R_n)} \E\bigl[|D_{x,y}^2\Tngamma|^5 \bigr].
$$
\begin{lemma}
\label{lem:c1c2bound}
Let the assumptions of Theorem \ref{t:clt.Tngamma} hold, and $\gamma \in (0,1/2)$. Set $\rho_n := ne^{-\zeta(d-1) R_n/2}.$ For $p \geq 1$, define $\mathcal{G}_p^{(1)}$ and $\mathcal{G}_p^{(2)}$ respectively to be the set of all trees on $\{0\} \cup [p]$ and $\{-1,0\} \cup [p]$. For a tree 
$T \in \mathcal{G}_p^{(1)}$, let $d'_0, d'_1,\ldots, d'_p$ be the degree sequence, and similarly, for a tree $T \in \mathcal{G}_p^{(2)}$, let $d'_{-1},d'_0,\ldots,d'_p$ be the degree sequence. Then, we have the following bounds for the constants defined above :
\begin{align}
c_{1,n} &=  O(\rho_n^{5(k-1)} c'_{1,n}), \label{e.c1bound} \\
c_{2,n} &= O(\rho_n^{5(k-2)} c'_{2,n}), \label{e.c2bound}
\end{align}
where 
\begin{align*} 
c'_{1,n}&:= \max_{\substack{p = k-1,\ldots,5(k-1), \\ T \in \mathcal{G}_p^{(1)}}}  e^{\zeta (d-1) d_0^\prime \gamma R_n/2} \prod_{i=1}^p\angamma(d'_i), \\
c'_{2,n} &:= \max_{\substack{p = k-2,\ldots,5(k-2), \\ T \in \mathcal{G}_p^{(2)}}}e^{\zeta (d-1)(d_{-1}' + d_0')\gamma R_n/2}\prod_{i=1}^p\angamma(d'_i). 
\end{align*}
\end{lemma}
\begin{lemma} \label{lem:c3bound}
In the notation of Lemma \ref{lem:c1c2bound}, define 
$$
c_{3,n} := \int_0^{\gamma R_n}\int_{C_d} \E\bigl[|D_x\Tngamma|^3 \bigr] \barrhona(t) \pi(\ta)  \md \ta \md t
$$
(the point $x$ is represented in its hyperbolic polar coordinate $(t,\ta)$).  Then, for $0 < \gamma < 1/2$, 
\begin{equation*}
c_{3,n} = O(\rho_n^{3(k-1)} c'_{3,n}),
\end{equation*}
where
$$ 
 c'_{3,n} := \max_{\substack{p = k-1,\ldots,3(k-1), \\ T \in \mathcal{G}_p^{(1)}}} \prod_{i=0}^p \angamma(d'_i). 
$$
\end{lemma}

\begin{remark}  \label{rem:slow.growth}
\noindent (i) An essential observation here is that one can make the growth rate of $c_{i,n}'$'s as slow as one likes by choosing $\gamma$ small enough. 
To make this a little more clear, assume, for simplicity, that $2\alpha /\zeta$ is not an integer. Then, as $n\to\infty$, 
$$
c_{1,n}^\prime \sim \max_{\substack{p = k-1,\ldots,5(k-1), \\ T \in \mathcal{G}_p^{(1)}}} \prod_{j=1}^p\,  \Bigl| \, (d-1) (\alpha-\zeta d_j'/2) \, \Bigr|^{-1} e^{\zeta (d-1)\bigl[  d_0^\prime + \sum_{i=1}^p \bigl( d_i^\prime -2\alpha /\zeta \bigr)_+ \bigr] \gamma R_n /2}.
$$ 
Note also that $\rho_n\to c \in (0,\infty]$ implies $R_n = O (\log n)$. Therefore, for any $\epsilon >0$, there exists $\gamma_0>0$ such that for all $0<\gamma<\gamma_0$, we have that $c_{1,n}' = o(n^\epsilon)$. The same claim holds for $c_{2,n}', c_{3,n}'$ as well. \\

\noindent (ii) The separation of $c_{i,n}$'s into $c'_{i,n}$'s and $\rho_n$ terms is because $c'_{i,n}$'s depend on $\gamma$, whereas $\rho_n$ does not. \\

\noindent (iii) In many applications to euclidean stochastic geometric functionals (see \citep[Section 7]{last:peccati:schulte:2016}), $c_{1,n},c_{2.n}$ are actually shown to be bounded whereas in our case $c_{1,n},c_{2.n}$ can be unbounded and this is an important reason why obtaining optimal Berry-Esseen bounds in our normal approximation result will be challenging.
\end{remark}
 In what follows, we use $g = \gngamma$ (see \eqref{eqn:gngamma}). An important step in our proof of the above lemmas is that $\Tngamma$ is a $U$-statistic and its Malliavin derivatives have a very neat form as follows : For two distinct points $x, y \in \Bdzeta$, 
\begin{align}
D_{x}\Tngamma &=  \sum_{\ell =1}^k \sum_{(X_1,\ldots, X_{k-1}) \in \cP^{k-1}_{n,\neq}}g(\underset{\hspace{-10pt} \ell}{X_1,\ldots,x,\ldots, X_{k-1}}), \label{e:diffustat1} \\
D^2_{x, y}\Tngamma &=  \sum_{1\leq \ell_1 < \ell_2 \leq k} \sum_{(X_1,\ldots, X_{k-2}) \in \cP^{k-2}_{n,\neq}}g(\underset{\hspace{-10pt} \ell_1 \, \, \  \ \    \ \, \, \, \, \, \, \, \, \hspace{0pt}\ell_2}{X_1,\ldots,x,\ldots,y,\ldots,X_{k-2}}) \label{e:diffustat2} \\
&+ \sum_{1\leq \ell_2 < \ell_1 \leq k} \sum_{(X_1,\ldots, X_{k-2}) \in \cP^{k-2}_{n,\neq}}g(\underset{\hspace{-10pt} \ell_2 \, \, \  \ \    \ \, \, \, \, \, \, \, \, \hspace{0pt}\ell_1}{X_1,\ldots,y,\ldots,x,\ldots,X_{k-2}}), \notag
\end{align}
where $\ell_i$'s denote the positions of the corresponding coordinates. For a proof, see Lemma 3.5 in \citep{reitzner:schulte:2013}. 
\begin{proof}[Proof of Lemma \ref{lem:c1c2bound}]
Fix $x \in \Bdzeta$. For $p=k-1,\dots,5(k-1)$, let $\Sigma_{5(k-1),p}$ denote the set of all surjective maps from $[5(k-1)]$ to $[p]$.  
From \eqref{e:diffustat1}, we have that
\begin{align*}
\bigl| D_{x}\Tngamma\bigr|^5 & \leq  C^* \sum_{\ell =1}^k \Bigl(\sum_{(X_1,\ldots, X_{k-1}) \in \cP^{k-1}_{n,\neq}}g(\underset{\ell}{X_1,\ldots,x,\ldots, X_{k-1}})\Bigr)^5 \\
& =  C^* \sum_{\ell =1}^k \sum_{p=k-1}^{5(k-1)} \frac{1}{p!}\sum_{\sigma \in \Sigma_{5(k-1),p}}\sum_{(X_1,\ldots, X_p) \in \cP^{p}_{n,\neq}} \prod_{i=1}^5g\bigl(\underset{\hspace{25pt}\ell}{X_{\sigma((i-1)(k-1)+1)},\ldots,x,\ldots, X_{\sigma(i(k-1))}}\bigr). 
\end{align*}
It is possible that under some surjections $\sigma$, the coordinates in  
$\bigl(\sigma((i-1)(k-1)+1),\ldots,\sigma(i(k-1))\bigr)$ may repeat for some $i$, but in such cases, $g = 0$ by definition; thus, $\Sigma_{5(k-1), p}$ in the last expression can be replaced with 
$$
\Sigma_{5(k-1),p}^* = \bigl\{\sigma \in \Sigma_{5(k-1), p}: \sigma((i-1)(k-1)+1),\ldots,\sigma(i(k-1)) \text{ are distinct for all } i=1,\dots,5   \bigr\}.
$$
Now, let us fix $\ell = 1$, without loss of generality, and $p \in \{k-1,\ldots,5(k-1)\}$, $\sigma \in \Sigma_{5(k-1),p}^*$, and then, we shall bound 
\begin{align}
A_{n,p,\sigma} := \E\Bigl[ \sum_{(X,\ldots, X_p) \in \cP^{p}_{n,\neq}} \prod_{i=1}^5 g\bigl(x,X_{\sigma((i-1)(k-1)+1)},\ldots, X_{\sigma(i(k-1))} \bigr) \Bigr]. \label{e:Anpsigma}
\end{align}
Let $G_\sigma$ be a simple graph on $\{0\}\cup [p]$ with the edge-set defined as follows : for every $i=1,\dots,5$, define $\bigl(\sigma((i-1)(k-1)+j_1),\sigma((i-1)(k-1)+j_2) \bigr) \in E_{\sigma}$ if $(j_1+1,j_2+1)$ is an edge in $\Gamma_k$. Similarly we say that $\bigl(0,\sigma((i-1)(k-1)+j)\bigr) \in E_{\sigma}$ if $(1,j+1)$ is an edge in $\Gamma_k$. Then, setting $X_0 = x$, we have that $G_{\sigma}$ is the graph counted by the summand in \eqref{e:Anpsigma}, i.e.,
\begin{equation}
\label{e.graph_count}
\prod_{i=1}^5 g\bigl(x,X_{\sigma((i-1)(k-1)+1)},\ldots, X_{\sigma(i(k-1))} \bigr) = \prod_{(i,j) \in G_{\sigma}}\one \bigl\{ d(X_i,X_j) \leq R_n \bigr\}.
\end{equation}

Now, the surjectivity of $\sigma$ implies that $G_\sigma$ is connected, and thus, we can always find a spanning tree $G_{\sigma}'$ of $G_\sigma$ on $\{  0\}\cup [p]$. Let $d_0', \dots, d_p'$ be the degree sequence of $G_\sigma'$ and $E'_\sigma$ be the edge set of $G_\sigma'$.  By the definition of $A_{n,p,\sigma},G_{\sigma}, G'_{\sigma}$ together with the monotonicity of $\C$ and the above identity, we have that
\begin{align}
A_{n,p,\sigma} &= \E\Bigl[\, \C\bigl(G_{\sigma},HG^{(\gamma)}_n(R_n;\alpha,\zeta)\bigr)\Bigr] 
\leq \E\Bigl[\, \C\bigl(G_{\sigma}',HG^{(\gamma)}_n(R_n;\alpha,\zeta)\bigr)\Bigr] \label{e:Anpsigma.bound}\\
&= n^p \P \bigl(\, d(X_i,X_j) \leq R_n, \ (i,j)\in E'_\sigma, \ T_i \leq \gamma R_n, \, i=1,\dots,p, \, t_0 \leq \gamma R_n  \bigr),\notag
\end{align}
where $t_0=R_n-d(0,x)$ is deterministic, and the Palm theory (Lemma \ref{l:palm1}) is applied at the last equality. Note that $|E'_\sigma|=p$. 

Proceeding as in the derivation of \eqref{e.An}, while noting that $T_i+T_j\leq R_n-\omega_n$ always holds, since we are taking $\gamma<1/2$, we derive that 
\begin{align*}
A_{n,p,\sigma} &\leq n^p \ingam dt_1 \cdots \ingam dt_p\, \one \{ t_0\leq \gamma R_n \} \prod_{(i,j)\in E'_\sigma} \P \bigl( d(X_i,X_j)\leq R_n\, |\, \bt \bigr)\barrhona (\bt) \\
&\sim \Bigl( \frac{2^{d-1}}{\kappa_{d-2}} \Bigr)^p\, \alpha^p n^p e^{-\zeta (d-1)p R_n/2} e^{\zeta (d-1) d_0' t_0/2}\one \{ t_0\leq \gamma R_n \}\prod_{i=1}^p\angamma (d_i')\\
&= O \bigl( \rho_n^p e^{\zeta (d-1) d_0' \gamma R_n/2} \prod_{i=1}^p\angamma(d'_i) \bigr).
\end{align*}

Now, we may conclude that
\begin{equation}  \label{e:c1.bound}
c_{1,n}  =  O \biggl( \max_{\substack{p=k-1,\dots,5(k-1), \\T \in \mathcal{G}_p^{(1)}}} \rho_n^p e^{\zeta (d-1) d_0' \gamma R_n/2} \prod_{i=1}^p\angamma(d'_i) \biggr). 
\end{equation}
If $\rho_n \to \infty$, then clearly, $\rho_n^p = O(\rho_n^{5(k-1)})$ for all $k-1 \leq p \leq 5(k-1)$, and hence, the bound \eqref{e.c1bound} holds trivially. Else, $\rho_n \to c \in (0,\infty)$, but we can still easily get \eqref{e.c1bound}.
\vspace{10pt}

Now, we shall show the bound for $c_{2,n}$ in \eqref{e.c2bound}. 
For $p \in \{ k-2,\dots,5(k-2) \}$, let 
$$
\Sigma_{5(k-2),p}^* = \bigl\{\sigma \in \Sigma_{5(k-2), p}: \sigma((i-1)(k-2)+1),\ldots,\sigma(i(k-2)) \text{ are distinct for all } i=1,\dots,5   \bigr\}. 
$$
Then, the task of bounding $\E\bigl[ \bigl|D_{x,y}\Tngamma\bigr|^5\bigr]$ is again reduced to that of bounding 
$$
B_{n,p,\sigma} := \E\Bigl[ \sum_{(X_1,\ldots, X_p) \in \cP^{p}_{n,\neq}} \prod_{i=1}^5g\bigl( x,y, X_{\sigma((i-1)(k-2)+1)},\dots, X_{\sigma(i(k-2))} \bigr) \Bigr]
$$ 
for all $p \in \{k-2,\ldots,5(k-2)\}$ and $\sigma \in \Sigma_{5(k-2),p}^*$. 

Setting $X_{-1} = x, X_0 = y$ respectively, we can define a simple graph $G_\sigma$ on $\{ -1,0 \} \cup [p]$ as in \eqref{e.graph_count} such that  
$$
\prod_{i=1}^5g\bigl( x,y, X_{\sigma((i-1)(k-2)+1)},\dots, X_{\sigma(i(k-2))} \bigr) = \prod_{(i,j) \in G_{\sigma}}\one \bigl\{d(X_i,X_j) \leq R_n\bigr\}.
$$
Let $G_\sigma'$ be a spanning tree of $G_\sigma$, which exists as $G_\sigma$ is connected by the surjectivity of $\sigma$. Let $d^\prime_{-1},d^\prime_0,\ldots,d^\prime_{p}$ be the respective degrees of vertices $\{-1,0\} \cup [p]$ in $G'_{\sigma}$, and $E'_\sigma$ be its edge set. Note that $|E'_\sigma|=p+1$.

Then, setting $t_{-1} = R_n-d(0,x)$ and $t_0=R_n-d(0,y)$ (which are deterministic), we again derive that
\begin{align}
B_{n,p,\sigma} &\leq \E\Bigl[\, \C\bigl(G_{\sigma}',HG^{(\gamma)}_n(R_n;\alpha,\zeta)\bigr)\Bigr] \label{e:Bnpsigma}\\
&= n^p \P \bigl(\, d(X_i,X_j) \leq R_n, \ (i,j)\in E'_\sigma, \ T_i \leq \gamma R_n, \, i=1,\dots,p, \, t_{-1}, t_0 \leq \gamma R_n  \bigr).\notag
\end{align}

We basically proceed again as in the derivation of \eqref{e.An} by conditioning on the $T_i$'s, but here, we need to account for extra complication, i.e.,  whether $(-1,0) \in E'_{\sigma}$ or not. If $(-1,0) \in E'_\sigma$, then whether $X_{-1}=x$ and $X_0=y$ are connected is purely deterministic and the randomness arises only in the remaining $p$ edges. In that case, the rightmost term in \eqref{e:Bnpsigma} is bounded by 
\begin{align}
&n^p \ingam dt_1 \cdots \ingam dt_p\, \one \{ t_{-1}, t_0\leq \gamma R_n \} \prod_{(i,j)\in E'_\sigma\setminus \{(-1,0) \}} \hspace{-10pt}\P \bigl( d(X_i,X_j)\leq R_n\, |\, \bt \bigr)\barrhona (\bt) \label{e:Bnpsigma.A}\\
&\sim C^* n^p e^{-\zeta (d-1)p R_n/2+2^{-1} \zeta (d-1) [ (d_{-1}' -1)t_{-1} + (d_0' - 1)t_0 ]}  \one \{ t_{-1}, t_0\leq \gamma R_n \}\prod_{i=1}^p\angamma (d_i') \notag \\
&= O \bigl( \, \rho_n^p e^{\zeta (d-1) (d_{-1}'+d_0') \gamma R_n/2-\zeta(d-1)\gamma R_n} \prod_{i=1}^p\angamma(d'_i) \, \bigr).\notag
\end{align}

On the other hand, if $(-1,0)\notin E'_\sigma$, the randomness arises in all the $p+1$ edges of $E'_\sigma$. Then, 
the rightmost term in \eqref{e:Bnpsigma} is bounded by 
\begin{align}
&n^p \ingam dt_1 \cdots \ingam dt_p\, \one \{ t_{-1}, t_0\leq \gamma R_n \} \prod_{(i,j)\in E'_\sigma} \hspace{-10pt}\P \bigl( d(X_i,X_j)\leq R_n\, |\, \bt \bigr)\barrhona (\bt) \label{e:Bnpsigma.B}\\
&\sim C^* n^p e^{-\zeta (d-1)(p+1) R_n/2 + 2^{-1} \zeta (d-1) (d_{-1}'t_{-1}+d_0' t_0 )}  \one \{ t_{-1}, t_0\leq \gamma R_n \}\prod_{i=1}^p\angamma (d_i') \notag \\
&= O \bigl( \, \rho_n^p e^{-\zeta(d-1)R_n/2+\zeta (d-1) (d_{-1}'+d_0') \gamma R_n/2} \prod_{i=1}^p\angamma(d'_i) \, \bigr).\notag
\end{align}
Combining \eqref{e:Bnpsigma.A} and \eqref{e:Bnpsigma.B}, we conclude that
$$
B_{n,p,\sigma} = O \bigl( \, \rho_n^p e^{\zeta (d-1) (d_{-1}'+d_0') \gamma R_n/2} \prod_{i=1}^p\angamma(d'_i) \, \bigr).
$$
Now proceeding as in the derivation of \eqref{e.c1bound} (see below \eqref{e:c1.bound}), we get \eqref{e.c2bound}.
\end{proof}
\begin{proof}[Proof of Lemma \ref{lem:c3bound}]
From the same reasoning as Lemma \ref{lem:c1c2bound}, it suffices to bound 
\begin{equation*}  
C_{n,p,\sigma} := \int_0^{\gamma R_n}\int_{C_d}\E \Bigl[   \sum_{(X_1,\ldots, X_p) \in \cP^{p}_{n,\neq}} \prod_{i=1}^3g\bigl(x,X_{\sigma((i-1)(k-1)+1)},\ldots, X_{\sigma(i(k-1))}\bigr) \Bigr] \barrhona(t) \pi(\ta) \md \ta \md t
\end{equation*}
for every $p \in \{ k-1,\dots,3(k-1) \}$ and $\sigma \in \Sigma_{3(k-1),p}^*$. 

Let $G_\sigma$ be the same simple graph on $\{0\} \cup [p]$ as that constructed in the proof of Lemma \ref{lem:c1c2bound} $(i)$ (see \eqref{e.graph_count}), for which $x$ is identified as a vertex ``$0$". 
Once again, let $G'_{\sigma}$ be a spanning tree of $G_{\sigma}$, and $d_0',\dots, d_p'$ is the degree sequence and $E'_\sigma$ is the edge set of $G_\sigma'$. It follows from the monotonicity of $\C$ and the Palm theory that
\begin{align}
C_{n,p,\sigma} &\leq \E\Bigl[\, \C\bigl(G_{\sigma}',HG^{(\gamma)}_n(R_n;\alpha,\zeta)\bigr)\Bigr] \label{e:prob.c3} \\
&= n^p \P \bigl(\, d(X_i,X_j) \leq R_n, \ (i,j)\in E'_\sigma, \ T_i \leq \gamma R_n, \, i=0,\dots,p \, \bigr) \notag
\end{align}
For further calculation, we note that $X_0$ in \eqref{e:prob.c3} is random, whereas $X_0=x$ in \eqref{e:Anpsigma.bound} was purely deterministic. Taking into consideration such a difference and using Theorem \ref{t:expectation.Tngamma} with $k=p+1$, we have
$$
C_{n,p,\sigma} \leq C^* \rho_n^p \prod_{i=0}^p \angamma (d_i^\prime) = O \Bigl( \rho_n^{3(k-1)}  \prod_{i=0}^p \angamma (d_i^\prime)\Bigr).
$$
Finally, taking maximum, we can complete the proof.
\end{proof}

\begin{lemma}  \label{lem:W1W2bound}
Let the assumptions of Theorem \ref{t:clt.Tngamma} hold. For $0 < \gamma < 1/2$ and $0 < a < 1$, we have
\begin{align*}
& \int_{[0,\gamma R_n]^3 \times C_d^3} [ \P(D^2_{x_1,x_3}\Tngamma \neq 0)\P(D^2_{x_2,x_3}\Tngamma \neq 0) ]^a \prod_{i=1}^3\barrhona(t_i)\pi(\ta_i) \md t_i \md \ta_i = O\bigl( e^{-\zeta(d-1)(1-2\gamma)R_n}\bigr), \\
& \int_{[0,\gamma R_n]^2 \times C_d^2} [ \P(D^2_{x_1,x_2}\Tngamma \neq 0)]^a \prod_{i=1}^2\barrhona(t_i)\pi(\ta_i) \md t_i \md \ta_i =   O\bigl( e^{-\zeta(d-1)(1-2\gamma)R_n/2}\bigr),
\end{align*}
where we identify $x_i$ with their hyperbolic polar coordinates $(t_i,\ta_i).$ 
\end{lemma}
We need the following fact for the proof of the lemma. For $0 < \gamma <1/2$, let $HG(\X)$ be a hyperbolic geometric graph on a point set $\X \subset D_{\gamma}(R_n)$, connecting any two points within a distance $R_n$. Suppose that $y_1,y_2 \in \X$ are connected by a path of length $\ell \geq 1$ in $HG(\X)$, and their hyperbolic distances from the boundary given by $t_i = R_n-d(0,y_i)$, $i=1,2$ satisfy $t_1,t_2 \leq \gamma R_n$. For the relative angle $\ta_{12}$ between $y_1$ and $y_2$, we claim that 
\begin{equation}
\label{e.l-conn-pts}
\ta_{12} \leq \bigl(1 +o(1)\bigr) 2 \ell e^{-\zeta(1-2\gamma)R_n/2}, \ \ \ n\to\infty,
\end{equation}
uniformly for $t_1,t_2 \leq \gamma R_n$. 

The proof of \eqref{e.l-conn-pts} can be done inductively. For $\ell=1$, set $\hat \ta_{12} = \bigl( e^{-2\zeta (R_n-t_1)} + e^{-2\zeta (R_n-t_2)} \bigr)^{1/2}$ as in Lemma \ref{l:angle}. Since $t_1 + t_2 \leq 2\gamma R_n < R_n-\omega_n$, we get 
$$
\hat \ta_{12} =o\bigl(e^{-\zeta(R_n-t_1-t_2)/2}\bigr) \to 0, \ \ \  n\to\infty, 
$$
as in the proof of Lemma \ref{l:int.angle}. 

If $\ta_{12} \leq \hat \ta_{12}$, then $\ta_{12} \leq \bigl(1+o(1)\bigr)e^{-\zeta (1-2\gamma) R_n/2}$ and so \eqref{e.l-conn-pts} holds for $\ell=1$. If $\ta_{12}\gg \hat \ta_{12}$, Lemma \ref{l:angle} yields the following : Uniformly for $t_i \leq \gamma R_n$, $i=1,2$, we have that
\begin{align*}
R_n &\geq d(y_1,y_2) = 2R_n -(t_1+t_2) + \frac{2}{\zeta} \log \sin \Bigl(\frac{\ta_{12}}{2} \Bigr) + o(1) \\
&\geq 2(1-\gamma) R_n +  \frac{2}{\zeta} \log \sin \Bigl(\frac{\ta_{12}}{2} \Bigr) + o(1),  \ \ \ n\to\infty. 
\end{align*}
Equivalently, we have that, uniformly for $t_i \leq \gamma R_n$, $i=1,2$, 
$$
\ta_{12} \leq \bigl( 1 + o(1) \bigr) 2 e^{-\zeta(1-2\gamma)R_n/2}, \ \ \ n\to\infty.
$$
Hence, in either case, the claim for $\ell = 1$ follows. 

Now, suppose that the claim holds for $\ell-1$. Then, if $y_1,y_2$ have a path of length $\ell$, there exists a $y_0$ such that $y_1$ and $y_0$ have a path of length $\ell-1$, and $y_0$ and $y_2$ have a path of length $1$. Denoting the corresponding relative angles as $\theta_{10}$ and $\theta_{02}$, we see that 
$\theta_{12} \leq \theta_{10} + \theta_{02}$; hence, the proof can be completed by the induction hypothesis. 
\begin{proof}[Proof of Lemma \ref{lem:W1W2bound}]
Fix $x_1,x_2,x_3$ such that  $t_i \leq \gamma R_n$ for $i =1,2,3$. From \eqref{e:diffustat2},  $D^2_{x_1,x_3}\Tngamma \neq 0$ implies that there is a path of length at most $diam(\Gamma_k)$ (i.e., diameter of the graph) from $x_1$ to $x_3$ in the hyperbolic random geometric graph on $\bigl( \cP_n \cap D_{\gamma}(R_n) \bigr) \cup \{x_1,x_3\}$ with radius of connectivity $R_n$. Thus, from \eqref{e.l-conn-pts} and $diam(\Gamma_k) \leq k$, we have that 
$$
\ta_{13} \leq \bigl(1 +o(1) \bigr)2ke^{-\zeta(1-2\gamma)R_n/2}, \ \ \ n\to\infty.
$$
Therefore, 
\begin{align*}
&\int_{[0,\gamma R_n]^3 \times C_d^3} [ \P(D^2_{x_1,x_3}\Tngamma \neq 0)\P(D^2_{x_2,x_3}\Tngamma \neq 0) ]^a \prod_{i=1}^3\barrhona(t_i)\pi(\ta_i) \md t_i \md \ta_i \\
&\leq C^* \int_{[0,\gamma R_n]^3 \times C_d^3} \one \bigl\{  \ta_{j3} \leq 2ke^{-\zeta (1-2\gamma)R_n/2}, \, j=1,2 \bigr\} \prod_{i=1}^3\barrhona(t_i)\pi(\ta_i) \md t_i \md \ta_i \\
&= C^* \P \bigl( \Ta_{j3} \leq 2k e^{-\zeta (1-2\gamma)R_n/2}, \, j=1,2, \ T_i \leq \gamma R_n, \ i=1,2,3 \bigr),
\end{align*}
where $\Ta_{j3}$ denotes the relative angle between $X_j$ and $X_3$, $j=1,2$, and $T_i=R_n-d(0,X_i)$ for $i=1,2,3$.

Now, the probability of the last term equals 
\begin{equation}  \label{e:joint.rela.angle}
\int_{[0,\gamma R_n]^3} \prod_{j=1}^2 \P \bigl( \Ta_{j3} \leq 2ke^{-\zeta (1-2\gamma)R_n/2}\, | \, t_1,t_2,t_3 \bigr)\prod_{i=1}^3\barrhona(t_i)\, \md t_i. 
\end{equation}
Using the density \eqref{e:relangle} of a relative angle, it is easy to see that 
$$
\prod_{j=1}^2 \P \bigl( \Ta_{j3} \leq 2ke^{-\zeta (1-2\gamma)R_n/2}\, | \, t_1,t_2,t_3 \bigr) \sim \biggl( \frac{(2k)^{d-1}}{(d-1)\kappa_{d-2}} \biggr)^2 e^{-\zeta (d-1) (1-2\gamma) R_n}, \ \ n\to\infty,
$$
uniformly for $t_i \leq \gamma R_n$, $i=1,2,3$. 
It now follows from Lemma \ref{l:pdf} $(ii)$ that \eqref{e:joint.rela.angle} is asymptotically equal to 
$$
C^* e^{-\zeta(d-1) (1-2\gamma)R_n} \left(  \ingam e^{-\alpha (d-1)t} dt\right)^3 = O\bigl( e^{-\zeta (d-1) (1-2\gamma) R_n} \bigr). 
$$
This proves the first result in the lemma. By completely the same argument, we can get the second relation in the lemma. 
\end{proof}
We now put together all the bounds and prove our main central limit theorem. 
\begin{proof}[Proof of Theorem \ref{t:clt.Tngamma}]
In order to apply Theorem \ref{thm:last2016}, we take
$$ 
F = \Tngamma, \ \ \ \ \  \lambda(\md x) = n \barrhona(t) \pi(\ta)\one \{ t \leq \gamma R_n\}\, dt d\ta, 
$$ 
where we have, once again, represented $x$ in its hyperbolic polar coordinate $(t,\ta)$.   From \eqref{e:variance.Tngamma}, we have
$$
\V(\Tngamma) = \Omega \biggl( n \rho_n^{2(k-1)}\angamma(2d_{(k)}) \prod_{i=1}^{k-1} \angamma(d_{(i)})^2 \biggr).
$$
Relying on the lower bound for variance above, along with the bounds from Lemmas \ref{lem:c1c2bound}, \ref{lem:c3bound}, \ref{lem:W1W2bound}, and the definition of $\rho_n$, we obtain
\begin{align*}
W_1 &\leq C^* n^{-1/2} \frac{(c'_{1, n}c'_{2, n})^{1/5}e^{\gamma \zeta (d-1) R_n}}{\angamma(2d_{(k)}) \prod_{i=1}^{k-1} \angamma(d_{(i)})^2} , \\[5pt]
W_2 &\leq C^* n^{-1/2} \rho_n^{-1} \frac{(c'_{2, n})^{2/5}e^{\gamma \zeta (d-1) R_n}}{\angamma(2d_{(k)}) \prod_{i=1}^{k-1} \angamma(d_{(i)})^2} , \\[5pt]
W_3 &\leq C^* n^{-1/2} \frac{c'_{3,n}}{\angamma(2d_{(k)})^{3/2} \prod_{i=1}^{k-1} \angamma(d_{(i)})^3},
\end{align*}
and from Theorem \ref{thm:last2016}, we know that
\begin{equation*}
d_W\left(\frac{\Tngamma - \E(\Tngamma)}{\sqrt{\V(\Tngamma)}},N\right) \leq W_1 + W_2 + W_3.
\end{equation*}
By the definitions of $\angamma$, $c'_{1,n},c'_{2,n}$, and $c'_{3,n}$ (see Lemmas \ref{lem:c1c2bound} and \ref{lem:c3bound}), along with the claim in Remark \ref{rem:slow.growth}, we have that for any $a < 1/2$, we can choose $\gamma_0$ so small that $W_1 + W_2 + W_3 = O(n^{-a})$ as $n \to \infty$, for all $0 < \gamma < \gamma_0$. This proves the Wasserstein bound in  \eqref{e:dW.Tngamma}.

To show the Kolmogorov bound in \eqref{e:dW.Tngamma}, again using the bounds in Theorem \ref{thm:last2016}, along with Lemmas \ref{lem:c1c2bound}, \ref{lem:c3bound}, and \ref{lem:W1W2bound}, and the variance lower bound above, we derive that
\begin{align*}
W_4 &\leq  C^* n^{-1/2} \biggl[ \, \frac{(c'_{1,n})^{3/5}}{\angamma(2d_{(k)})^{3/2} \prod_{i=1}^{k-1} \angamma(d_{(i)})^3} +  \frac{(c'_{1,n})^{4/5}}{\angamma(2d_{(k)})^2 \prod_{i=1}^{k-1} \angamma(d_{(i)})^4} \, \biggr] , \\[5pt]
W_5 &\leq  C^* n^{-1/2} \frac{(c'_{1,n})^{2/5}}{\angamma(2d_{(k)}) \prod_{i=1}^{k-1} \angamma(d_{(i)})^2} , \\[5pt]
W_6 &\leq  C^* (n\rho_n)^{-1/2} \frac{e^{\gamma \zeta (d-1) R_n/2}}{\angamma(2d_{(k)}) \prod_{i=1}^{k-1} \angamma(d_{(i)})^2} \bigl((c'_{1,n}c'_{2,n})^{1/5} + \rho_n^{-1} (c'_{2,n})^{2/5}\bigr).
\end{align*}
From Theorem \ref{thm:last2016}, we have
\begin{equation*}
d_K\left(\frac{\Tngamma - \E(\Tngamma)}{\sqrt{\V(\Tngamma)}},N\right) \leq W_1 + W_2 + W_3 + W_4 + W_5 + W_6,
\end{equation*}
and hence, for any $a < 1/2$, we can choose $\gamma_0$ small enough such that the Kolmogorov bound in \eqref{e:dW.Tngamma} holds for all $0< \gamma < \gamma_0$. 

In order to show \eqref{e:clt.tn}, let us assume $\alpha / \zeta > d_{(k)}$. First, choose $0<\gamma<1/2$ such that 
$$
d_W\left(\frac{\Tngamma - \E(\Tngamma)}{\sqrt{\V(\Tngamma)}},N\right) \to 0, \ \ \text{as } n\to\infty.
$$ 
Setting $\Ungamma = S_n - \Tngamma$, we write
\[ \frac{S_n - \E(S_n)}{\sqrt{\V(S_n)}} = \sqrt{\frac{\V(\Tngamma)}{\V(S_n)}} \times \frac{\Tngamma - \E(\Tngamma)}{\sqrt{\V(\Tngamma)}}  + \frac{\Ungamma - \E(\Ungamma)}{\sqrt{\V(S_n)}}. \]
From \eqref{e:variance.Sn}, we have that $\V(\Tngamma) \sim  \V(S_n)$ as $n \to \infty$. Since the central limit theorem holds for $\Tngamma$, the first term converges in distribution to $N$ as $n \to \infty.$  Now, from \eqref{e:variance.Ungamma}, we know that $\V(\Ungamma) / \V(S_n) \to 0$ as $n \to \infty$, and hence, by Chebyshev's inequality, the second term converges to $0$ in probability. Thus, applying Slutsky's theorem, we obtain the central limit theorem for $S_n$ as required. 
\end{proof}

\section{{\bf Appendix}}
\label{sec:Appendix}

\subsection{Palm theory for Poisson point processes}
\label{sec:Palm}~\\

This result is known as the \textit{Palm theory} of Poisson point processes (see Section 1.7 in \citep{penrose:2003}), which is applied a number of times throughout the proof. 
\begin{lemma}
\label{l:palm1}
Let $X_1,X_2,\dots$ be $\bbr^d$-valued iid random variables with density $f$, and $\Pn = \{ X_1,\dots,X_{N_n} \}$ be the Poisson point process on $\bbr^d$, where $N_n$ is a Poisson random variable with mean $n$ and is independent of $(X_i)$.  Let $h, h_i : (\bbr^d)^k \to \bbr, i=1,2,\dots$ be bounded measurable functions vanishing on diagonals of $(\bbr^d)^k$, that is, $h(x_1,\dots,x_k) = h_i (x_1,\dots,x_k)=0$ whenever at least two of the $x_j$'s are equal. Then, 
\begin{equation} \label{e:palm1}
\E \Bigl[ \, \sum_{\X \in \mathcal P_{n,\neq}^k} h(\X) \Bigr] = n^k\, \E \bigl( h(X_1,\ldots,X_k) \bigr)
\end{equation}
where $\mathcal P_{n,\neq}^k$ is defined in \eqref{e:distinctpts}. 

Moreover, let $q\geq 2$ and $p \in \{ k, k+1,\dots, qk \}$ and $\Sigma_{qk, p}$ be a collection of surjective maps from $[qk]$ to $[p]$. Then,
\begin{align}
&\E \Bigl[ \ \sum_{\X_1 \in \mathcal P_{n,\neq}^k} \cdots \sum_{\X_q \in \mathcal P_{n,\neq}^k} \prod_{j=1}^q h_j(\X_j)\, \one \bigl\{ \, |   \cup_{i=1}^q\X_i | = p \bigr\} \Bigr] = \frac{n^p}{p!}\,  \sum_{\sigma \in \Sigma^*_{qk,p}} \E \Bigl[\  \prod_{j=1}^q h_j(X_{\sigma ( (j-1)k+1 )}, \dots, X_{\sigma (jk)})\, \Bigr], \label{e:palm2}
\end{align}
where $\Sigma^*_{qk,p}$ is a subset of $\Sigma_{qk,p}$ such that $\sigma((j-1)k + 1),\ldots,\sigma(jk)$ are distinct for all $j = 1,\ldots,q$.
\end{lemma}

\begin{proof}
Since \eqref{e:palm1} is a special case of \eqref{e:palm2}, we only prove the latter. For $p \in \{ k,\dots, qk \}$, we have
\begin{align*}
&\sum_{\X_1 \in \mathcal P_{n,\neq}^k} \cdots \sum_{\X_q \in \mathcal P_{n,\neq}^k} \prod_{j=1}^q h_j(\X_j)\, \one \bigl\{ \, |   \cup_{i=1}^q\X_i | = p \bigr\} \\
&= \frac{1}{p!}\, \sum_{\sigma \in \Sigma_{qk,p}} \sum_{(X_1,\dots,X_p)\in \mathcal P_{n,\neq}^p}\prod_{j=1}^q h_j(X_{\sigma ( (j-1)k+1 )}, \dots, X_{\sigma (jk)}). 
\end{align*}
Since $h_j$'s vanish on the diagonals of $(\bbr^d)^k$, one can replace $\Sigma_{qk,p}$  with $\Sigma_{qk,p}^*$. 

Conditioning on $N_n$, we have 
\begin{align*}
&\E \biggl[ \ \sum_{(X_1,\dots,X_p)\in \mathcal P_{n,\neq}^p}\prod_{j=1}^q h_j(X_{\sigma ( (j-1)k+1 )}, \dots, X_{\sigma (jk)}) \biggr] \\
& =  \sum_{m=p}^\infty \E \biggl[ \  \sum_{(X_1,\dots,X_p)\in \mathcal I_{m,\neq}^p}\prod_{j=1}^q h_j(X_{\sigma ( (j-1)k+1 )}, \dots, X_{\sigma (jk)})  \biggr]\, \frac{e^{-n}n^m}{m!}, 
\end{align*}
where $\mathcal I_m = \{ X_1,\dots,X_m \}$ and 
\begin{align*}
\mathcal I_{m,\neq}^p &= \bigl\{  (X_{i_1},\dots,X_{i_p}) \in \mathcal I_m^p: i_j \in \{1,\dots,m\}, \  i_j \neq i_\ell \ \text{for } j \neq \ell \bigr\}. 
\end{align*}
Since 
\begin{align*}
&\E \biggl[ \  \sum_{(X_1,\dots,X_p)\in \mathcal I_{m,\neq}^p}\prod_{j=1}^q h_j(X_{\sigma ( (j-1)k+1 )}, \dots, X_{\sigma (jk)})  \biggr]  = \frac{m!}{(m-p)!} \E \biggl[ \  \prod_{j=1}^q h_j(X_{\sigma ( (j-1)k+1 )}, \dots, X_{\sigma (jk)})  \biggr],
\end{align*}
the result follows by a simple calculation. 
\end{proof}

\subsection{Comparison to subgraph counts of Euclidean random geometric graphs}
\label{sec:comparison}~\\

In this section, we shall briefly sketch analogous asymptotic results for subgraph counts of random geometric graphs when the underlying metric is Euclidean. For simplicity, we shall restrict ourselves to the case $\alpha = \zeta = 1$ and $R_n = 2(d-1)^{-1}\log(n/\nu)$ for some $\nu > 0$ as in Section \ref{sec:specialcase}. In other words, we are considering the uniform distribution on the Poincar\'e ball, and further, by Corollary \ref{cor:specialcase}, we have that $\E(E_n) = \Theta(n)$, where $E_n$ denotes the number of edges in  $HG_n(R_n;1,1)$. Further, from the metric equivalence of the hyperbolic metric with the Euclidean metric on a compact ball of the Poincar\'e disk, the below asymptotics also give asymptotic growth rates for $HG_n(R;\alpha,\zeta)$ for any $\alpha, \zeta > 0$. 

There are two possible ways in which one can consider Euclidean analogues of our results in Corollary \ref{cor:specialcase}.
\begin{enumerate}
\item Consider Poisson($n$) points distributed uniformly in a sequence of growing Euclidean balls of radius $r_n$ ($r_n \to \infty$) and connect any two points within a distance $r_n$ - ({\em dense regime}). We shall call this graph $EG_{1,n}$. 
\item Consider Poisson($n$) points distributed uniformly in a sequence of Euclidean balls of radius $r_n$ and connect points within a distance $s_n$ such that the expected number of edges grows linearly in $n$ - ({\em thermodynamic regime}).  We shall call this graph $EG_{2,n}$.
\end{enumerate}
\begin{figure}[!htbp]
\centering
\caption{Simulations of $EG_{1,100}$ with $\pi r^2_{100} = 100$ and $EG_{2,500}$ with $\pi r^2_{500} = 500, s_{500} =1$ for $d = 2$}
\includegraphics[width=2.5in,height=2.4in]{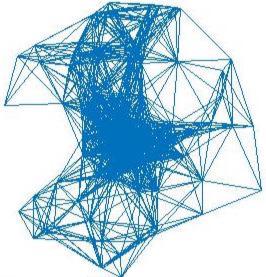} \hspace*{1.5cm}
\includegraphics[width=2.5in,height=2.4in]{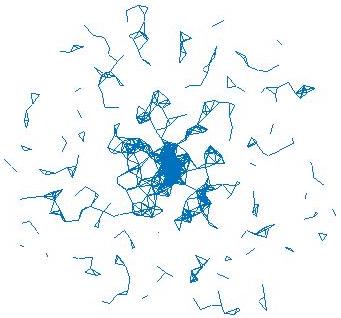} 
\label{fig:ergg}
\end{figure}
See Figure \ref{fig:ergg} for particular simulations of these two Euclidean graphs. Unlike hyperbolic random geometric graphs, these two regimes are distinct for Euclidean random geometric graphs, clarifying the choice of terminology for regimes (dense and thermodynamic). We shall only give a sketch of the calculations but refer the reader to \citep[Chapter 3]{penrose:2003} for details.

Fix $d \geq 2$ and let us denote the collection of Poisson($n$) points distributed uniformly in $B_{r_n}(0) \subset \bbr^d$ (i.e., the $d$-dimensional ball of radius $r_n$ centred at origin) as $\X_n$. Then, $EG_n(r_n,s_n)$ denotes the graph with vertex set $\X_n$ and edges between $X_i,X_j \in \X_n$ such that $|X_i - X_j| \leq s_n$, where $|\cdot|$ is the Euclidean metric. Under this notation, $EG_{1,n} = EG_n(r_n,r_n)$ and $EG_{2,n} = EG_n(r_n,s_n)$ for a suitable choice of $s_n$ satisfying a condition about linear growth of expected edges. Let $\Gamma$ be a connected graph on $k$ vertices,  and by $J_{i,n}(\Gamma), i=1,2$, we denote the number of copies of $\Gamma$ in $EG_{i,n}, i =1,2$, which can be defined similarly to sub-tree counts in \eqref{e:tngamma}. Now, by using the Palm formula for Poisson point processes (see Lemma \ref{l:palm1}), we have that
\[ \E(J_{1,n}(\Gamma)) = n^k \frac{C_1}{r_n^{dk}} \int_{ B_{r_n}(0)^k} \prod_{(i,j) \in \Gamma}\one\bigl\{|x_i-x_j| \leq r_n \bigr\}\md x_1 \ldots \md x_k = \Theta(n^k).\]
Since the order of $\E(J_{1,n}(\Gamma))$ is the same as that of the complete subgraph on $\X_n$, we call it a {\em dense regime}. The point we wish to observe is that the choice of $r_n$ and the degree sequence of the subgraph $\Gamma$ are irrelevant to the growth of $\E(J_{1,n}(\Gamma))$. Alternatively, only the number of vertices of $\Gamma$ determines the asymptotics. This is quite unlike the asymptotics for hyperbolic random geometric graphs in Corollary \ref{cor:specialcase}.

As in the above case, assuming $s_n = o(r_n)$, we can derive that $\E(J_{2,n}(\Gamma)) = \Theta\bigl(n^k (\frac{s_n}{r_n})^{d(k-1)}\bigr)$. Thus, the expected number of edges in $EG_{2,n}$ (i.e., $\E(J_{2,n}(K_2))$, where $K_2$ is the connected graph on two vertices) is $\Theta\bigl(n^2(\frac{s_n}{r_n})^d\bigr)$, and so, if we choose $s_n = n^{-1/d}r_n$ with $r_n \to r \in (0,\infty]$, we get that $\E(J_{2,n}(K_2)) = \Theta(n)$. This is called the thermodynamic regime (see \citep[Chapter 3]{penrose:2003}), since the expected average degree (or empirical count of neighbours) is asymptotically constant. This is true for the hyperbolic random geometric graph for the regime of Corollary \ref{cor:specialcase}. In contrast to the hyperbolic random geometric graph, we see from the above calculation that the asymptotics of $\E(J_{2,n}(\Gamma))$ is again independent of the degree sequence of the subgraph $\Gamma$ and the choice of $r_n$. 

In conclusion, either of the Euclidean analogues to the hyperbolic random geometric graph are markedly different in the sense that the degree sequence of the subgraph count does not affect the asymptotic first order growth. Though we do not discuss second order or more finer results, one can find them in \citep[Chapter 3]{penrose:2003} for Euclidean random geometric graphs or  in \citep{penrose2013,bobrowski:mukherjee:2015} for random geometric graphs on compact manifolds, and the broad message remains unchanged. 

\section*{Acknowledgements} The work benefitted from the visit of both the authors' to Technion, Israel and the authors are thankful to their host Robert Adler for the same. DY is also thankful to Department of Statistics at Purdue University for hosting him.  DY also wishes to thank Subhojoy Gupta for some helpful discussions on hyperbolic geometry. 

\bibliographystyle{plain}
\bibliography{HGG}


\end{document}